\documentclass[11pt,a4paper,reqno]{amsart}

\usepackage[includehead,includefoot,margin=25mm]{geometry}
\usepackage{amsmath,amsthm,amssymb}
\usepackage{times}
\usepackage{enumerate}
\usepackage{amscd}
\usepackage{comment}
\usepackage{color}
\usepackage{tikz-cd}
\usepackage{graphpap}
\usepackage[normalem]{ulem}

\newtheorem{thm}{Theorem}[section]
\newtheorem{cor}[thm]{Corollary}
\newtheorem*{cor*}{Corollary}
\newtheorem{lem}[thm]{Lemma}

\newtheorem{prop}[thm]{Proposition}

\theoremstyle{definition}
\newtheorem{defin}[thm]{Definition}
\newtheorem{rem}[thm]{Remark}
\newtheorem{exa}[thm]{Example}

\numberwithin{equation}{section}

\newcommand{\p}{{\mathbb{P}}}

\newcommand{\ddd}{\delta}

\newcommand{\NN}{{\mathbb N}}
\newcommand{\CC}{{\mathbb C}}
\newcommand{\RR}{{\mathbb R}}
\newcommand{\QQ}{{\mathbb Q}}

\newcommand{\be}{\begin{enumerate}}
	\newcommand{\ee}{\end{enumerate}}
\newcommand{\beq}{\begin{equation}}
	\newcommand{\eeq}{\end{equation}}
\DeclareMathOperator{\id}{id}
\DeclareMathOperator{\SL}{SL}
\DeclareMathOperator{\Aut}{Aut}

\DeclareMathOperator{\GL}{GL}
\DeclareMathOperator{\Lie}{Lie}

\DeclareMathOperator{\Der}{Der}

\DeclareMathOperator{\Bir}{Bir}
\DeclareMathOperator{\Cent}{Cent}
\DeclareMathOperator{\Jonq}{Jonq}
\DeclareMathOperator{\Tr}{Tr}

\newcommand{\bbmat}{\begin{bmatrix}}
	\newcommand{\ebmat}{\end{bmatrix}}
\newcommand{\bsmat}{\begin{smallmatrix}}
	\newcommand{\esmat}{\end{smallmatrix}}

\DeclareMathOperator{\Aff}{Aff}
\DeclareMathOperator{\SAff}{SAff}

\newcommand{\C}{\mathbb{C}}
\newcommand{\Z}{\mathbb{Z}}
\DeclareMathOperator{\PGL}{PGL}
\DeclareMathOperator{\PSL}{PSL}
\newcommand{\Q}{\mathbb{Q}}
\newcommand{\A}{\mathbb{A}}

\renewcommand{\phi}{\varphi}

\DeclareMathOperator{\Pic}{Pic}

\newcommand{\GA}[0]{\ensuremath{\mathbb{G}_{\mathrm{a}}}}
\newcommand{\GM}[0]{\ensuremath{\mathbb{G}_{\mathrm{m}}}}
\newcommand{\KK}[0]{\ensuremath{\mathbb{C}}}
\newcommand{\spec}[0]{\ensuremath{\operatorname{Spec}}}
\newcommand{\TT}[0]{\ensuremath{T}}
\newcommand{\OO}[0]{\ensuremath{\mathcal{O}}}
\newcommand{\RT}[0]{\ensuremath{\mathcal{R}}}
\newcommand{\pr}[0]{\ensuremath{\operatorname{pr}}}
\newcommand{\NS}[0]{\ensuremath{\operatorname{NS}}}

\begin{document}

	\title[Characterization of affine varieties with a torus action]{Characterization of
		affine surfaces with a torus action \\ by their automorphism groups}
	
	\author{Alvaro Liendo}
	\author{Andriy Regeta}  \author{Christian Urech}
	\thanks{Partially supported by Fondecyt projects 1160864 and 1200502.}
	\address{Instituto de Matem\'atica y F\'\i sica \\ Universidad de Talca \\ Casilla 721, Talca, Chile}
	\email{aliendo@inst-mat.utalca.cl}
	
	\address{Institut f\"{u}r Mathematik \\ Friedrich-Schiller-Universit\"{a}t Jena, \\
		Jena 07737, Germany}
	\email{andriyregeta@gmail.com}
	\thanks{Supported by the SNF, project number P2BSP2\_175008.}
	\address{EPFL SB MATH \\ 
		Station 8 \\
		1015 Lausanne, Switzerland}
	\email{christian.urech@gmail.com}

	\begin{abstract}
		In this paper we prove that if two normal affine surfaces
		$S$ and $S'$ have isomorphic automophism groups, then every
		connected algebraic group acting regularly and faithfully on
		$S$ acts also regularly and faithfully on $S'$.          Moreover, if
		$S$ is non-toric, we show that the dynamical type of a
		1-torus action is preserved in presence of an additive group
		action.
		We also show that complex affine toric surfaces are determined by
		the abstract group structure of their regular automorphism
		groups in the category of complex normal affine surfaces
		using properties of the Cremona group. As a generalization
		to arbitrary dimensions, we show that complex affine toric
		varieties, with the exception of the algebraic torus, are
		uniquely determined in the category of complex affine normal
		varieties by their automorphism groups seen as ind-groups.

	\end{abstract}
	
	\maketitle
	
	\tableofcontents

	\section{Introduction}
In the whole paper we work over the field $\CC$ of complex numbers and
varieties are considered to be irreducible.  Let $\TT$ be the
algebraic torus, i.e.\,$\TT\simeq \GM^n$, where $\GM=(\KK^*,\cdot)$ is
the multiplicative group of the base field $\C$.  A $\GM$-variety is a
variety endowed with a regular $\GM$-action. A toric variety is a
normal algebraic variety endowed with a $\TT$-action having a Zariski
dense open orbit.

Let $G$ be an algebraic group acting faithfully and regularly on an
affine surface $S$. This induces an embedding of $G$ into
$\Aut(S)$. Subgroups of $\Aut(S)$ of this form are called {\it
  algebraic subgroups} (see Section~\ref{indgrps} for details). In
this paper we look at the question, in as far this algebraic structure
of subgroups of $\Aut(S)$ is encoded by the structure of $\Aut(S)$ as
an abstract group. As a first main result we show that abstract group
isomorphisms preserve algebraic groups.
	
\begin{thm} \label{main} %
  Let $S$ and $S'$ be normal affine surfaces and let
  $\varphi\colon\Aut(S)\to\Aut(S')$ be an isomorphism of groups. If
  $G$ is a connected algebraic subgroup of $\Aut(S)$ that is not
  unipotent, then $\varphi(G)\subset\Aut(S')$ is an algebraic subgroup
  isomorphic to $G$ as an algebraic group.
\end{thm}
	
The restriction on unipotent subgroups in the theorem cannot be
removed as shown by Example~\ref{example-unipotent}, where we exhibit a
surface $S$ all of whose automorphisms are unipotent such that $\Aut(S)$ is
isomorphic to $(\CC,+)$ as group. This surface allows us to produce a
counter-example to Theorem~\ref{main} in the case where $G$ is
unipotent. However, let $S$ and $S'$ be normal affine surfaces and let
$\varphi\colon\Aut(S)\to\Aut(S')$ be an isomorphism of groups. If
$\Aut(S)$ contains a unipotent subgroup $U$, then $\Aut(S')$ also
contains an algebraic subgroup isomorphic to $U$ (see
Remark~\ref{unipotentrem}).
	
Furthermore, $\GM$-surfaces come in three types with respect to the
dynamical behavior of the $\GM$-action: \emph{elliptic} corresponding
to the case where there is an attractive fixed point, \emph{parabolic}
where there are infinitely many fixed points, and \emph{hyperbolic}
where there are finitely many non-attractive fixed points. For a
non-toric surface $S$ admitting a $\GM$-action, the dynamical type is
an invariant of the surface and does not depend on the choice of the
$\GM$-action (\cite[Corollary~4.3]{MR2126657}). In the case of
surfaces whose automorphism group contains algebraic subgroups
isomorphic to $\GA$ and $\GM$, the dynamical type is preserved by
isomorphisms of the automorphism groups.
	
\begin{thm}\label{dynamical}
  Let $S$ and $S'$ be normal affine $\GM$-surfaces with $S$
  non-toric. Assume that there is a non-trivial regular action of
  $\GA$ on $S$. If there exists a group isomorphism
  $\varphi\colon \Aut(S)\rightarrow\Aut (S')$, then the $\GM$-surfaces
  $S$ and $S'$ are of the same dynamical type.
\end{thm}
	
The presence of a $\GA$-action in Theorem~\ref{dynamical} cannot be
removed as is shown by Example~\ref{example-elliptic} and
Example~\ref{example-hyperbolic}, where we exhibit a hyperbolic and an
elliptic surface both having automorphism group isomorphic to
$\GM$. Note that most toric surfaces can have $\GM$-actions of all the
dynamical types, so the dynamical type is not an invariant of the
surface anymore.

In a next step, we apply our techniques to the following question: is
a toric variety uniquely determined by its automorphism group? This
question can be seen in the context of the Erlangen program of Felix
Klein, in which he suggested to understand geometrical objects through
their groups of symmetries (\cite{MR1510801}).  Note that in general
it is impossible to characterize affine algebraic varieties by their
groups of regular automorphisms, since most of them have a trivial
automorphism group. However, in the case of toric surfaces the
automorphism group is well-studied and large enough to determine the
variety.
	
\begin{thm}\label{toricthm}
  Let $S_1$ be an affine toric surface and let $S_2$ be a normal
  affine surface. If $\Aut(S_1)$ and $\Aut(S_2)$ are isomorphic as
  groups, then $S_1$ and $S_2$ are isomorphic.
\end{thm}
	
An important class of elements in $\Aut(S)$, where $S$ is an affine
surface, are {\it algebraic} elements, i.e.\,elements that are
contained in an algebraic group (see Section \ref{algsubgrps}). The
main idea of the proof of Theorem~\ref{main} and
Theorem~\ref{toricthm} is to consider an automorphism of an affine
surface $S$ as an element of $\Bir(S)$, the group of birational
transformations of $S$. We show that an element in $\Bir(S)$ is
algebraic if and only if an iterate of it is divisible (Theorem
\ref{divisibility}). From this purely group theoretical
characterization we obtain that algebraic elements are preserved under
group homomorphisms and we are able to reconstruct the corresponding
surfaces.
	
Much less is known about the group structure of $\Bir(X)$ if $X$ is a
variety of dimension greater than two. Hence, we are not able to
generalize Theorem~\ref{main} and Theorem~\ref{toricthm} to arbitrary
dimensions. However, the automorphism group of an affine variety comes
with the additional structure of an ind-group (see Section
\ref{indgrps}). We show that the automorphism group as an ind-group
determines a toric variety in most of the cases:
	
\begin{thm}\label{mainhighdim}
  Let $X$ be an affine toric variety different from the algebraic
  torus, and let $Y$ be a normal affine variety. If $\Aut(X)$ and
  $\Aut(Y)$ are isomorphic as ind-groups, then $X$ and $Y$ are
  isomorphic.
\end{thm}
	
In fact, we show that toric varieties and their automorphism groups
are uniquely determined by the weights of their root subgroups (see
Section \ref{sec:higher}). In the case of finite dimensional algebraic
groups this is a property of reductive groups
\cite[Section~4.4]{MR1642713}.
	
Theorem \ref{mainhighdim} should be seen in the context of the results
from \cite{MR3738084} (see also \cite{regeta2017characterization}),
which show that the complex affine space is characterized by its
automorphism group seen as an ind-group in the category of not
necessarily irreducible affine varieties. The assumption that $X$ is
not an algebraic torus in Theorem \ref{mainhighdim} cannot be removed
as shown by Example~\ref{main2}.  We also remark that the normality
condition in Theorem~\ref{main} and Theorem~\ref{toricthm} cannot be
removed as shown in \cite[Theorem
2]{regeta2017characterization}. However, in the particular case of the
two-dimensional affine space, one can remove the normality hypothesis
(Theorem~\ref{A2}).  In Section~\ref{autaut} we will deduce from
Theorem~\ref{main} an algebraic analogue to a result of Filipkiewicz
(\cite{MR693972}) about isomorphisms between groups of diffeomorphisms
of manifolds without boundary.

\subsection*{Acknowledgements} The authors would like to warmly thank
J\'er\'emy Blanc, Michel Brion, Serge Cantat, Adrien Dubouloz, Hanspeter Kraft,
Alexander Pere\-pechko, and Mikhail Zaidenberg for many useful discussions, comments, and references.  Part of this
work was done during a stay of the three authors at IMPAN in
Warsaw. We would like to thank IMPAN and the organizers of the Simons
semester ``Varieties: Arithmetic and Transformations'' for the
hospitality.
	
\noindent This work was partially supported by the grant 346300 for
IMPAN from the Simons Foundation and the matching 2015-2019 Polish
MNiSW fund. The first author was supported by Fondecyt project number
1160864.  The second author was partially supported by SNF, project
number P2BSP2\_165390 and the third author was supported by the SNF,
project number P2BSP2\_175008.

	\section{Automorphisms and birational transformations of affine
		varieties}\label{Preliminaries}
	
	In this section we recall some results that we will need about
	automorphisms and birational transformations of normal affine
	varieties.
	
	\subsection{Ind-groups}\label{indgrps}
	
	The notion of an ind-group was introduced by Shafarevich, who called
	these objects infinite dimensional groups, see \cite{MR0485898}.  We
	refer to \cite{furter2018geometry} for an overview of the topic.
	
	\begin{defin} \label{ind-group} %
		An \emph{ind-variety} is a set $V$ together with an
		ascending filtration
		$V_0 \subset V_1 \subset V_2 \subset \ldots \subset V$ such that the
		following conditions are satisfied:
		\begin{enumerate}[$(1)$]
			\item $V = \bigcup_{k \geq 0} V_k$;
			\item each $V_k$ has a structure of an algebraic variety;
			\item for every $k \in \mathbb{Z}_{\geq 0}$, the embedding
			$V_k \subset V_{k+1}$ is closed in the Zariski-topology.
		\end{enumerate}
	\end{defin}
	Note  that in particular every algebraic variety $X$ is an ind-variety by setting $X_k:=X$ for all $k$. 
	A \emph{morphism} between ind-varieties $V = \bigcup_k V_k$ and
	$W = \bigcup_m W_m$ is a map $\phi\colon V \to W$ such that for any $k$
	there is an $m \in \mathbb{Z}_{\geq 0}$ such that $\phi(V_k) \subset W_m$ and
	such that the induced map $V_k \to W_m$ is a morphism of algebraic
	varieties. An \emph{isomorphism} of ind-varieties is defined in the usual way.
	An ind-variety $V$ can be equipped with a topology: a subset $S \subset V$ is open if $S_k := S \cap V_k \subset V_k$ is open for all $k$. A locally closed subset
	$S \subset V$ has the natural structure of an ind-variety and is called an
	{\it ind-subvariety}. An ind-variety $V$ is called \emph{affine} if all the $V_k$ are
	affine varieties.

	\begin{defin}
		An affine ind-variety $G$ is called  an \emph{ind-group} if the
		underlying set $G$ is a group such that the map $G \times G \to G$,
		defined by $(g,h) \mapsto gh^{-1}$, is a morphism of ind-varieties.
	\end{defin}
	
	Of course, ind-groups can be defined in a more general way, but since in this paper we only deal with ind-groups whose underlying ind-variety is affine, we go with this more restrictive definition. 
	
	A closed subgroup $H$ of $G$
	is again an ind-group under the closed
	ind-subvariety structure on $G$.  A closed subgroup
	$H$ of an ind-group $G$ is called an {\it algebraic subgroup} if $H$ is contained in some filter set $G_k$ of $G$. It then follows that the ind-group structure induces the structure of an algebraic group on $H$ and that the inclusion $H\hookrightarrow G$ is a homomorphism of ind-groups. Note that with our definition of ind-groups, algebraic subgroups are always linear.

	\begin{thm}[{\cite[Theorem 0.3.1]{furter2018geometry}}, {\cite[Theorem 3.7]{stampfli2013contributions}}]\label{ind-group-prop}
		Let $X$ be an affine variety. Then $\Aut(X)$ has the structure
		of an  ind-group such that for any algebraic group $G$, there is a natural correspondence between regular $G$-actions on $X$ and ind-group homomorphisms $G \to \Aut(X)$. 
	\end{thm}

In particular, if $G$ is an algebraic group acting regularly and faithfully on $X$, then we can consider $G$ as an algebraic subgroup of $\Aut(X)$. We will sometimes implicitly switch between these two points of view.

	The following observation will turn out to be useful:
	
	\begin{lem}\label{closed}
		Let $U \subset \Aut(X)$ be a commutative subgroup that coincides
		with its centraliser. Then $U$ is a closed subgroup of $\Aut(X)$.
	\end{lem}
	\begin{proof}
		Let $u\in U$ and define $G_u = \{ g \in \Aut(X) \mid g u = u g
		\}$.
		Since $ug=gu$ is a closed condition on each filter set, we obtain
		that $G_u \subset \Aut(X)$ is a closed subgroup. Hence,
		$\cap_{u \in U} G_u = U$ is closed in $\Aut(X)$.
	\end{proof}
	
	\subsection{The Zariski topology on $\Bir(X)$}\label{zariskitop}
	Let $X$ be a variety and denote by $\Bir(X)$ its group of birational
	transformations. Blanc and Furter show in \cite{MR3092478} that
	$\Bir(\p^n)$ is not an ind-group. However, it still comes with the
	so-called {\it Zariski topology}, which has been introduced by
	Demazure (\cite{MR0284446}).  Let $A$ be a variety and let
	\[
	f\colon A\times X\dashrightarrow A\times X
	\]
	be an $A$-birational map, i.e.\,$f$ is the identity on the first factor,
	that induces an isomorphism between open subsets $U$ and $V$ of
	$A\times X$ such that the projections from $U$ and from $V$ to $A$ are
	both surjective. From this definition it follows that each $a\in A$ defines an element in
	$\Bir(X)$ and we obtain a map $A\to\Bir(X)$. A map of this form is called a
	{\it morphism}. The {\it Zariski topology} is now defined to be the finest
	topology on $\Bir(X)$ such that all the morphisms $A\to\Bir(X)$ for all
	varieties $A$ are continuous with respect to the Zariski topology on
	$A$. For all $g\in\Bir(X)$ the maps $\Bir(X)\to\Bir(X)$
	given by $x\mapsto x^{-1}$, $x\mapsto g\circ x$ and
	$x\mapsto x\circ g$ are continuous.
	
	Assume that $X$ is the projective $n$-space $\p^n$. With respect to
	homogeneous coordinates $[x_0:\dots :x_n]$, an element $f\in\Bir(\p^2)$
	is given by $[x_0:\dots:x_n]\mapsto [f_0:\dots:f_n]$, where the
	$f_i\in \C[x_0,\dots,x_n]$ are homogeneous polynomials of the same
	degree $d$ without a non-constant common factor. We call $d$ the {\it degree} of
	$f$. 
	
	\subsection{Algebraic subgroups of $\Bir(X)$ and $\Aut(X)$}\label{algsubgrps}
	An {\it algebraic subgroup} of $\Bir(X)$ is the image of an
	algebraic group $G$ by a morphism $G\to\Bir(X)$ that is also
	an injective group homomorphism.  An element $g\in\Bir(X)$ is
	called {\it algebraic} if $g$ is contained in an algebraic subgroup of $\Bir(X)$. On the other hand, an
	element $g\in\Aut(X)$ is {\it algebraic} if it is contained in
	an algebraic subgroup $G$ of $\Aut(X)$ with respect to its ind-group structure. The following results, which follow from Weil's regularization theorem, show that the two notions coincide for automorphisms of affine varieties.
	
	\begin{prop}[{\cite[Theorem 2]{kraft2018regularization}}]\label{weilcor}
		Let $X$ be an affine variety and $\rho\colon G\to\Bir(X)$ a rational $G$-action on $X$. If there is a dense subgroup $H\subset G$ such that $\rho(H)\subset\Aut(X)$, then the action of $G$ on $X$ is regular.
	\end{prop}
	
	\begin{cor}\label{algelement}
		Let $X$ be an affine variety and $g\in\Aut(X)$. Then $g$ is algebraic in $\Bir(X)$ if and only if $g$ is algebraic in $\Aut(X)$.
	\end{cor}
	
	\begin{proof}
		If $g$ is algebraic in $\Aut(X)$, then $g$ is contained in an algebraic subgroup $G\subset\Aut(X)$, which is, by definition, also an algebraic subgroup of $\Bir(X)$. On the other hand, assume that $g$ is algebraic in $\Bir(X)$. Then $g$ is contained in an algebraic group $G\subset\Bir(X)$ that acts rationally on $X$. Hence $\overline{\langle g\rangle}\subset \Bir(X)$ is also an algebraic subgroup. By Proposition~\ref{weilcor}, $\overline{\langle g\rangle}$ acts regularly on $X$ and hence, $g$ is algebraic in $\Aut(X)$.
	\end{proof}

	\section{Characterization of algebraic elements 
		on surfaces}\label{divisible}
	
	\subsection{Divisibility in the Cremona group}
	
	Recall that an element $f$ in a group $G$ is called {\it divisible by}
	$n$ if there exists an element $g\in G$ such that $g^n=f$. An element
	is called {\it divisible} if it is divisible by all $n\in\Z^+$. We use divisibility in order to characterize algebraic elements in $\Bir(S)$ for surfaces $S$:
	
	\begin{thm}\label{divisibility}
		Let $S$ be a surface and $f\in\Bir(S)$. Then the following two conditions are equivalent:
		\begin{enumerate}
			\item there exists a $k>0$ such that $f^k$ is divisible;
			\item $f$ is algebraic.
		\end{enumerate}
	\end{thm}
	
	In order to prove Theorem \ref{divisibility}, we need some results
	from dynamics of birational transformations of surfaces. Let $H$ be an ample divisor class on $S$ and denote by $f^*H$ the total transform of $H$ under $f\in\Bir(S)$. The {\it degree of $f$ with respect to $H$} is defined as 
	\[
	\deg_H(f)=f^*H\cdot H.
	\]
	If $f$ is an element in $\Bir(\p^2)$ and $H$ the class of a hypersurface, then
	$\deg_H(f)=\deg(f)$, the degree we have defined in Section \ref{zariskitop}.
	Let $H_1$ and $H_2$ be two different ample divisors. Then there exists a constant $C>0$ such that
	\[
	\frac{1}{C}\deg_{H_1}(f)\leq\deg_{H_2}(f)\leq C\deg_{H_1}(f)
	\]
	for all $f\in\Bir(S)$ (see for example \cite[Theorem 1(2)]{bac_dang}).
	
	If we fix an ample divisor $H$ on $S$, we can associate to each $f\in\Bir(S)$ its degree sequence $\{\deg_H(f^n)\}_{n\in\Z_+}$. The growth of the degree sequence of a birational transformation carries information about its dynamical behavior. The following Theorem~\ref{cremonathm} is crucial for the understanding of groups of birational transformations in dimension two. It has been developped in various papers. We refer to \cite[Theorem 4.6, Theorem 5.4]{cantatrev} for the version we state below and detailed references to its proof.
	
	\begin{thm}[Gizatullin; Diller, Favre; Blanc; Cantat
		(\cite{MR563788, diller2001dynamics, cantatrev, blanc2016dynamical})]\label{cremonathm}
		Let $S$ be a projective surface, $H$ an ample divisor on $S$ and $f\in\Bir(S)$  a birational transformation. Then we are in exactly one of the following cases:
		\begin{enumerate}
			\item the sequence
			$\{\deg_H(f^n)\}_{n\in\Z_+}$ is bounded, which is equivalent to $f$ being
			algebraic;
			\item $\deg_H(f^n)\sim cn$ for some constant $c>0$ and
			$f$ preserves a rational fibration;
			\item $\deg_H(f^n)\sim cn^2$ for some  constant $c>0$ and
			$f$ preserves an elliptic fibration;
			\item $\deg_H(f^n)\sim c\lambda^n $ for some constant $c>0$, where $\lambda$ is a Pisot or Salem
			number.
		\end{enumerate}
	\end{thm}

	Blanc and D\'eserti gave lower bounds for the constant $c$ appearing in the cases  (b) and (c)  of Theorem~\ref{cremonathm} if the surface $S$ is rational:
	
	\begin{thm}[{\cite[Theorem C]{MR3410471}}]\label{blancdeserti}
		Let $f\in\Bir(\p^2)$ and let $H$ be the divisor class of a line. Assume that $\deg_H(f^n)\sim cn$, then $c\geq 1/2$. 
	\end{thm}
	
	We also need the following:
	
	\begin{thm}[{\cite[Corollary~4.7]{blanc2016dynamical}}]\label{cantat}
		Let $S$ be a projective surface with an ample divisor $H$ and $f\in\Bir(S)$ such that $\deg(f^n)$ grows exponentially with $n$. Then the centralizer of $f$ equals $\left<f\right>$ up to
		finite index.
	\end{thm}
	
	Theorem \ref{blancdeserti} can be generalized to non-rational surfaces of negative Kodaira dimension:
	
	\begin{lem}\label{negkod}
		Let $S=C\times\p^1$, where $C$ is a smooth projective, non-rational curve. Then there exists an ample divisor class $H$ on $S$ such that for all $f\in\Bir(S)$ we are in one of the following cases:
		\begin{enumerate}
			\item the sequence $\{\deg_H(f^n)\}_{n\in\Z_+}$ is bounded and $f$ is
			algebraic;
			\item $\deg_H(f^n)\sim cn$ for some constant $c\geq 1/2$.
		\end{enumerate}
	\end{lem}
	
	To prove Lemma \ref{negkod}, we need some birational
	geometry. Let $f$ be a birational transformation of a
	projective surface $S$. Whenever we speak of base-points, we
	consider both, proper and infinitely near base-points.  A
	base-point $p$ of $f$ is called {\it persistent} if there
	exists an integer $N$ such that $p$ is a base-point of $f^k$
	for all $k\geq N$ but $p$ is not a base-point of $f^{-k}$ for
	any $k\geq N$. In \cite[Proposition~3.5]{MR3410471}, the authors show that $f$ has no persistent
	base-points if and only if $f$ is conjugate to an automorphism
	of a smooth projective surface. Another important fact from
	\cite{MR3410471} is the following:
	
	\begin{thm}[{\cite[Proposition~3.4]{MR3410471}}]\label{numberbp}
		Let $S$ be a smooth projective surface and $f\in\Bir(S)$. Denote by $\mathfrak{b}(f^n)$ the number of base-points of $f^n$. Then there exists a non-negative integer $\nu$ such that the set $\{\mathfrak{b}(f^n)-n\nu\mid n\geq 0\}\subset\Z$ is bounded. 
	\end{thm}
	
	\begin{proof}[Proof of Lemma \ref{negkod}]
		Since $C$ is not rational, $f$ preserves the $\p^1$-fibration given by the first projection.
		We have $\Pic(S)\simeq\Pic(C)\oplus\Pic(\p^1)$, where
		the embedding of $\Pic(C)$ into $\Pic(S)$ is given by
		the pullback of the first projection $\pi_1\colon S\to
		C$ and the embedding of $\Pic(\p^1)$ into $\Pic(S)$ by
		the pullback of the second projection $\pi_2\colon
		S\to \p^1$. Let $P\in\Pic(C)$ and $Q\in\Pic(\p^1)$ be
		the divisor class of a  single point
		in $C$ and $\p^1$, respectively and let $F_P:= \pi_1^*P$ and $L:= \pi_2^*Q$ in $\Pic(S)$. Define the ample divisor $H:= F_P+L$. 
		
		If $\{\deg_H(f^n)\}_{n\in\Z_+}$ is bounded, then $f$ is algebraic, by Theorem \ref{cremonathm}. Assume now that the degree-sequence $\{\deg_H(f^n)\}_{n\in\Z_+}$ is unbounded. 
		Since $f$ preserves the fibration given by the first projection, we have that $f^*F_P=F_{P'}$, where $P'\in \Pic(C)$ is the divisor class of another point in $C$. Moreover, we have $(f^n)^*L=L+D$, for some $D\in\Pic(C)$ and hence $(f^n)^*(F_P+L)=L+F_{P'}+D$. It follows that $\deg_H(f^n)=(f^n)^*(F_P+L)\cdot(F_P+L)=\deg(D)+2$ and therefore $\deg(D)=\deg_H(f^n)-2$. We obtain that the total transform $(f^n)^*L$ has self-intersection $((f^n)^*L)\cdot ((f^n)^*L)=2\deg(D)+2=2(\deg_H(f^n)-1)$. Since $f$ preserves the $\p^1$-fibration, all base-points have multiplicity one. The divisor class $L$ has self-intersection $0$, hence we obtain that $f^n$ must have $2(\deg_H(f^n)-1)$ base-points. By Theorem~\ref{numberbp}, the number of base-points of $f^n$ grows asymptotically like $Kn$ for some integer $K$, hence $\deg_H(f^n)$ grows asymptotically like $(K/2)n$. 
	\end{proof}

	\begin{lem}\label{glnzdiv}
		Let $n>0$ and $A\in\GL_n(\Z)$ an element such that $A^k$ is divisible for some $k\neq0$. Then $A$ is of finite order.
	\end{lem}
	
	\begin{proof}
		It is enough to show that there exists no divisible
		element of infinite order in $\GL_n(\mathbb{Z})$. Let
		$B\in\GL_n(\Z)$ be of infinite order. For a prime $p$
		let  $\varphi_p\colon \GL_n(\Z)\to\GL_n(\mathbb{F}_p)$
		be the homomorphism given by reduction modulo $p$. We
		may choose $p$ such that $B$ is not contained in the
		kernel of $\varphi_p$. The image $\varphi_p(B)$ is
		then not divisible by $k:=
		|\GL_n(\mathbb{F}_p)|$. Hence, $B$ is not divisible by
		$k$. 
	\end{proof}
	
	\begin{lem}\label{divaut}
		Let $X$ be a projective variety and $f\in\Aut(X)$ an element such that $f^k$ is divisible for some $k\in\Z_{>0}$. Then $f$ is algebraic. 
	\end{lem}
	
	\begin{proof}
		The kernel of the action of the group $\Aut(S)$ on the Neron-Severi lattice $\NS(X)$ is an algebraic group (see for example \cite{notesBrion}). Hence we obtain a group homomorphism $\rho\colon \Aut(X)\to \GL_n(\Z)$, where $n$ is the Picard rank of $X$ and $\ker(\rho)$ is an algebraic group. So if $f^k\in\Bir(X)$ is a divisible element, then Lemma \ref{glnzdiv} shows that $f^n$ is contained in the kernel of $\rho$ for some $n>0$. In particular, $f$ is algebraic.
	\end{proof}
	
	\begin{lem}\label{poskod}
		Let $S$ be a projective surface of non-negative Kodaira dimension and $f\in\Bir(S)$. If $f^k$ is divisible for some $k\neq 0$, then $f$ is algebraic. 
	\end{lem}
	
	\begin{proof}
		Since the Kodaira dimension of $S$ is non-negative, there
		exists a unique minimal model $S'$ in the birational
		equivalence class of $S$ and we have
		$\Bir(S)=\Bir(S')=\Aut(S')$ (see
		\cite[Corollary~10.22]{MR1805816}). The proof now follows from Lemma \ref{divaut}. 
	\end{proof}
	
	\begin{lem}\label{halphen}
		Let $f\in\Bir(\p^2)$ be an element such that $\deg(f^n)\sim cn^2$. Then $f^k$ is not divisible for all $k>0$.
	\end{lem}
	
	\begin{proof}
		If $\deg(f^n)\sim cn^2$, then $f$ is conjugate to an automorphism of a Halphen surface, i.e. a smooth projective surface $S$ such that the linear system $\vert -mK_S\vert$ is one-dimensional, has no fixed component, and is base-point free (see for example \cite[Section 2]{MR2904576} for details). In this case, $\vert -mK_S\vert$ defines the unique pencil of elliptic curves that is preserved by $f$. Let $g\in\Bir(S)\simeq\Bir(\p^2)$ and $n\in\Z_{>0}$ be such that $g^n=f$. Since $\deg(g^n)\sim c'n^2$ for some constant $c'>0$, the transformation $g$ and therefore $g^n$ preserve a unique pencil of elliptic curves, which therefore has to be $\vert -mK_S\vert$. In particular, since $-mK_S$ has self-intersection $0$, we obtain that $g$ preserves $K_S$  and is therefore an automorphism of $S$. In other words, all elements in $\Bir(S)$ that divide $f$ are automorphisms of $S$. Hence, for all $k>0$ we have that $f^k$ is divisible in $\Bir(S)$ if and only if $f$ is divisible in $\Aut(S)$. So if $f^k$ is divisible, then $f^k$ is algebraic by Lemma \ref{divaut}, which is a contradiction to $\{\deg(f^n)\}$ being unbounded.
	\end{proof}
	
	\begin{lem}\label{divalg}
		Let $S$ be a projective surface and $f\in\Bir(S)$ an
		element such that $f^k$ is divisible for some $k>
		0$. Then $f$ is algebraic in
		$\Bir(S)$.
	\end{lem}
	
	\begin{proof}
		First we consider the case, where $S$ is rational. Let
		$f\in\Bir(S)$ and $H$ an ample divisor on $S$. If $f$ is of
		finite order, then $f$ is algebraic. So we may assume that $f$ is of
		infinite order. We consider the four cases, given by
		Theorem~\ref{cremonathm}. If $\{\deg_H(f^n)\}$ is bounded,
		Theorem~\ref{cremonathm} implies that $f$ is algebraic. If
		$\deg(f^n)\sim cn$, assume that there is a $g\in\Bir(\p^2)$
		and a $l\geq 0$ such that $g^l=f$. It follows that
		$\deg_H(g^n)\sim \frac{c}{l} n$. By
		Theorem~\ref{blancdeserti}, the constant $c$ has to be at
		least $1/2$, so $l\leq 2c$ and  $f$ is only divisible by finitely many
		integers $l$.The case $\deg_H(f^n)\sim cn^2$ is not possible by Lemma~\ref{halphen}.   Finally, if
		$\deg(f^n)\sim c\lambda^n$, we observe that every element
		that divides $f$ centralizes $f$. By Theorem
		\ref{cantat}, there are only finitely many elements
		$g\in\Bir(\p^2)$ that divide $f$ and the same holds for any iterate of $f$.
		
		If $S$ is non-rational and of Kodaira dimension $-\infty$, we use  Lemma \ref{negkod} and proceed with a similar argument as in the rational case.
		
		If $S$ is of Kodaira dimension $\geq 0$, the result follows from Lemma \ref{poskod}.
	\end{proof}
	
	An algebraic
	group $H$ is called {\it anti-affine} if
	$\mathcal{O}(H)=\C$. 	 If $G$ is an arbitrary connected
	algebraic group, there exists a central anti-affine
	group $G_{ant}\subset G$ such that $G/G_{ant}$ is
	linear  (see \cite{MR2488561}). Denote by $\GA$ the additive group $(\C,+)$ of the field of complex numbers. 
	
	\begin{lem}\label{lineardiv}
		Let $G$ be an algebraic group and $g\in G$. Then there exists a $k>0$ such that $g^k$ is divisible.
	\end{lem}
	
	\begin{proof}
		We may assume that $G$ is connected. If $G$ is linear, consider the Zariski-closure $A:=\overline{\langle g\rangle}$, which is a commutative subgroup of $G$. Hence, $A\simeq \mathbb{G}_m^{n_1}\times \mathbb{G}_a^{n_2}\times H$, for some $n_1\geq 0$, $n_2\in\{0,1\}$ and a finite group $H$. Let $k$ be the order of $H$. Then $g^k$ is contained in $U\simeq \mathbb{G}_m^{n_1}\times \mathbb{G}_a^{n_2}\times \{\id\}\subset A$, which is a group in which every element is divisible.
		
		Let $G_{ant}\subset G$ be a central anti-affine group such that $G/G_{ant}$ is linear. By \cite[Lemma
		1.6]{MR2488561}, every element in $G_{ant}$ is
		divisible. Let now $g\in G$ be arbitrary. As
		$G/G_{ant}$ is linear there exists a $k$ such that the
		class of $[g^k]$ is divisible in $G/G_{ant}$, i.e.\,for
		every $n\geq 0$ there exists an element $f\in G$ such
		that $f^nh=g^k$ for some $h\in G_{ant}$. Since
		$G_{ant}$ is divisible, there is a $h'\in G_{ant}$
		satisfying $h'^n=h$ and hence $(fh')^n=g^k$,
		i.e.\,$g^k$ is divisible.
	\end{proof}

	\begin{proof}[Proof of Theorem \ref{divisibility}]
		By Lemma \ref{divalg}, if for  $f\in\Bir(S)$ there exists a $k>0$ such that $f^k$ is divisible, then $f$ is algebraic. On the other hand, let $f\in\Bir(S)$ be an algebraic element. Then $f$ is contained in an algebraic subgroup $G\subset\Bir(S)$. By Lemma~\ref{lineardiv}, we obtain that $f^k$ is divisible for some $k>0$. 
	\end{proof}

	With the help of Theorem \ref{divisibility} we are now able to prove one of the main tools of this paper:

	\begin{prop}\label{alg}
		Let $S_1$ and $S_2$ be normal affine surfaces,
		$\phi\colon \Aut(S_1)\to\Aut(S_2)$ an abstract group  homomorphism and
		$g\in\Aut(S_1)$ an algebraic element. Then $\varphi(g)$ is an algebraic element
		in $\Aut(S_2)$.
	\end{prop}
	
	\begin{proof}
		The element $g$ is algebraic in $\Aut(S_1)$ and therefore a power of $g$
		divisible. Since divisibility is preserved by group homomorphisms,
		we obtain that $\phi(g)$ is divisible in $\Aut(S_2)$. Therefore, by
		Theorem~\ref{divisibility}, $\phi(g)$ is algebraic in
		$\Bir(Y)$. Corollary~\ref{algelement} implies that $\phi(g)$ is
		algebraic in $\Aut(S_2)$.
	\end{proof}

	\section{Root subgroups of affine varieties}\label{toric}

	\subsection{Root subgroups}

	In this section we describe  {\it root subgroups} of
	$\Aut(X)$ for a given affine variety $X$ with respect to a torus.
	
	\begin{defin}
		Let $\TT \subset \Aut(X)$ be a torus in $\Aut(X)$, i.e. a closed
		algebraic subgroup isomorphic to a torus.  A closed subgroup
		$U \subset \Aut(X)$ isomorphic to $\GA$ is called a {\it root
			subgroup} with respect to $\TT$ if the normalizer of $U$ in
		$\Aut(X)$ contains $\TT$.
		
		Since $\GA$ contains no non-trivial closed normal subgroups, every non-trivial regular action is faithful. Hence such a
		subgroup $U$ is equivalent to a non-trivial {\it normalized} $\GA$-action on $X$,
		i.e.\,a $\GA$-action on $X$ whose image in $\Aut(X)$ is normalized
		by $\TT$.
	\end{defin}
	
	Let $U\subset \Aut(X)$ be a root subgroup with respect to $\TT$. Since
	$\TT$ is in the normalizer, we can define an action
	$\varphi\colon\TT\rightarrow \Aut(U)$ of $\TT$ on $U$ given by
	$t.s=t\circ s\circ t^{-1}$ for all $t\in \TT$ and $s\in U$.
	Furthermore, since $\Aut(U)\simeq \GM$, such an action corresponds to a
	character of the torus $\chi\colon\TT\rightarrow \GM$, which does not
	depend on the choice of automorphism between $\Aut(U)$ and
	$\GM$. This character is
	called the \emph{weight character} of $U$. The algebraic subgroups $T$ and $U$ span an algebraic subgroup in $\Aut(X)$
	isomorphic to $\GA\rtimes_\chi T$.

	Assume that the algebraic torus $\TT$ acts linearly and regularly on a
	vector space $A$ of countable dimension. We say that $A$ is {\it
		multiplicity-free} if the weight spaces $A_{\chi}$ are all of
	dimension less or equal than one for every character
	$\chi\colon\TT\rightarrow \GM$ of the torus $\TT$. In our proof of
	Theorem \ref{mainhighdim}, we will use the following lemma that is due to
	Kraft:
	
	\begin{lem} [{\cite[Lemma~6.2]{MR3738084}}] \label{Kr15} %
		Let $X$ be a normal affine variety and let $\TT \subset \Aut(X)$ be
		a torus. If there exists a root subgroup $U \subset \Aut(X)$ with
		respect to $T$ such that $\mathcal{O}(X)^{U}$ is multiplicity-free,
		then $\dim \TT \le \dim X \le \dim \TT +1$.
	\end{lem}

	Additive group actions, or, equivalently, unipotent one-parameter subgroups on affine varieties can be described by a
	certain kind of derivations. We recall some of the basics here (see
	\cite{MR2259515} for details). Let $\lambda\colon\GA \to \Aut(X)$ be a
	$\GA$-action on an affine variety $X$. This action induces a
	derivation on the level of regular functions by
	\[\delta_\lambda\colon\OO(X)\rightarrow\OO(X),\qquad
	f\mapsto \left[\frac{d}{ds}\lambda{(s)}^*(f)\right]_{s=0},
	\]
where $\GA=\spec(\C[s])$.
	This derivation has the property that for every $f\in \OO(X)$ there exists an
	$\ell\in \NN$ with $\delta_\lambda^\ell(f)=0$. Derivations having this
	property are called {\it locally nilpotent}. Furthermore, every
	$\GA$-action on $X$ arises from such a locally nilpotent derivation
	$\delta$ and the $\GA$-action
	$\alpha_\delta\colon\GA\times X\rightarrow X$ is recovered from $\delta$
	via
	\[\left(\alpha_\delta(s)\right)^*\colon\OO(X)\rightarrow \OO(X)[s],
	\qquad
	f\mapsto\exp(s\delta)(f):=\sum_{i=0}^\infty\frac{s^i\delta^i(f)}{i!}\,.\]

	Let $\TT\subseteq \Aut(X)$ be an algebraic torus. The choice of such a
	$T$ is equivalent to fixing an $M$-grading on the ring $\OO(X)$ of
	regular functions, where $M$ is the character lattice of the torus. We
	follow the standard convention to consider $M$ as an abstract additive
	lattice and to denote the character corresponding to $m\in M$ by
	$\chi^m$.
	
	Recall that a linear map $\delta\colon A\rightarrow B$ between
	$M$-graded $\KK$-vector spaces is called {\it homogeneous} if there
	exists an $e\in M$ such that for every homogeneous element $f$ of
	degree $m$, the image $\delta(f)$ is homogeneous of degree $m+e$. We
	call the element $e\in M$ the {\it degree} of $\delta$ and denote it
	by $\deg\delta$.

	The next proposition states that root subgroups are in one to one
	correspondence with locally nilpotent derivations that are homogeneous
	with respect to the $M$-grading of $\OO(X)$. A proof can be found in
	\cite[Lemma~2]{MR2852493}.
	
	\begin{prop}\label{root-hom}
		Let $X$ be an affine variety and fix a torus $\TT\subseteq
		\Aut(X)$. The  unipotent one-parameter subgroup corresponding to a locally nilpotent derivation $\delta$ on $\OO(X)$ is
		a root subgroup with respect to $\TT$ if and only if it is homogeneous with respect to
		the $M$-grading on $\OO(X)$ given by $\TT$. The weight of the
		corresponding root subgroup is $\chi^{\deg\delta}$.
	\end{prop}
	
We will also use the following result from \cite[Lemma~3]{makar1998locally} and its generalization in \cite[Lem\-ma~1.10]{MR2657447}:

\begin{lem}\label{ML}
	If a variety $X$ admits a non-trivial $\GA$-action, then $\Aut(X)$ contains root subgroups with respect to {any} torus $T\subset\Aut(X)
$.\end{lem}

	\subsection{Root subgroups of affine toric varieties}
	
	An affine toric variety is a normal affine
	variety endowed with a
	faithful action of an algebraic torus $\TT$ that acts with an open
	orbit. An affine toric variety $X$ is called \emph{non-degenerate} if
	it has no torus factor, i.e.\,if it is not isomorphic to
	$Y\times \A^1_*$ for some variety $Y$, where
	$\A^1_*=\A^1\setminus \{0\}$.
	
	In this section we first recall the well known description of affine
	toric varieties by means of strongly convex rational polyhedral cones,
	details can be found in reference texts about toric geometry such as
	\cite{MR922894, MR1234037, MR2810322}.  Then we provide a description
	of root subgroups of the automorphism group of an affine toric variety.
	
	Let $M$ and $N$ be dual lattices of rank $n$ and consider the duality
	pairing $M\times N\rightarrow \Z$, defined by
	$(m,p)\mapsto \langle m,p\rangle=p(m)$.  Let $M_\RR=M\otimes_\Z \RR$
	and $N_\RR=N\otimes_\Z \RR$ be the corresponding real vector spaces
	and let $\TT$ be the algebraic torus
	$\TT=\spec\KK[M]=N\otimes_\Z\KK^*\simeq\GM^n$. With this choice, $M$ is
	the character lattice of $\TT$ and $N$ the lattice of 1-parameter
	subgroups of $\TT$.
	
	By a well known construction, affine toric varieties can be described
	via strongly convex rational polyhedral cones in the vector space
	$N_\RR$. Let $\sigma$ be a strongly convex rational polyhedral cone in
	$N_\RR$ and let $\KK[\sigma^\vee\cap M]$ be the semigroup algebra
	$\KK[\sigma^\vee\cap M]=\bigoplus_{m\in \sigma^\vee\cap M}\KK\chi^m$,
	where the multiplication rule is given by
	$\chi^m\cdot\chi^{m'}=\chi^{m+m'}$ and $\chi^0=1$.  In the following,
	we denote $\sigma^\vee\cap M$ by $\sigma^\vee_ M$.
	
	The main result about affine toric varieties is that
	$X_\sigma:=\spec \KK[\sigma^\vee_M]$ is an affine toric variety, where
	the comorphism
	$\alpha^*\colon\KK[\sigma^\vee_M]\rightarrow\KK[M]\otimes_\KK\KK[\sigma^\vee_M]$
	of the $\TT$-action is given by
	$\chi^m\mapsto\chi^m\otimes\chi^m$. Furthermore, every affine toric
	variety arises via this construction.
	
	We now describe root subgroups of the automorphism group of a toric
	variety. Let $\sigma\subseteq N_\RR$ be a strongly convex rational
	polyhedral cone. Following the usual convention, we identify a ray
	$\rho\subseteq\sigma$ with its shortest non-trivial vector, called its
	primitive vector. The set of all the rays of $\sigma$ is denoted by
	$\sigma(1)$.
	\begin{defin}
		We say that a lattice vector $\alpha\in M$ is a {\it root} of
		$\sigma$ if there exists $\rho_\alpha\in \sigma(1)$ such that
		$\langle \alpha,\rho_\alpha\rangle = -1$ and
		$\langle \alpha,\rho\rangle \geq 0$, for every $\rho\in\sigma(1)$
		different from $\rho_\alpha$. We call the ray $\rho_\alpha$ the
		\emph{distinguished ray} of the root $\alpha$. We denote by
		$\RT(\sigma)$ the set of all roots of $\sigma$ and by
		$\RT_\rho(\sigma)$ the set of all roots of $\sigma$ with
		distinguished ray $\rho$.
	\end{defin}
	
	Let $\alpha\in \RT(\sigma)$. One checks that the linear map given by
	$$\delta_{\alpha}\colon\KK[\sigma_M^\vee]\rightarrow\KK[\sigma_M^\vee],\quad \chi^m\mapsto
	\langle m,\rho_\alpha\rangle\cdot\chi^{m+\alpha}$$ is a homogeneous
	locally nilpotent derivation of the algebra
	$\KK[\sigma_M]$. Furthermore, it was proven implicitly in
	\cite{MR0284446} and explicitly in \cite[Theorem~2.7]{MR2657447} that
	every homogeneous locally nilpotent derivation of the algebra
	$\KK[\sigma_M^\vee]$ arises this way. We summarize these results in
	the following proposition.
	
	\begin{prop} \label{liendo} Let $X_\sigma$ be the affine toric variety
		given by a strongly convex rational polyhedral cone
		$\sigma\subset N_\RR$.  The root subgroups of $\Aut(X_\sigma)$ with
		respect to $\TT$ are in one to one correspondence with the roots of
		the cone $\sigma$. The correspondence is given by assigning to every
		$\alpha\in\RT(\sigma)$ the root subgroup whose homogeneous locally
		nilpotent derivation is $\delta_\alpha$ with weight $\chi^\alpha$.
	\end{prop}

	From Proposition~\ref{liendo} it follows that if the 
	root subgroups corresponding to $\delta_\alpha$ and $\delta_{\alpha'}$ have the same distinguished ray, the corresponding derivations have the same kernel, and therefore the two root subgroups
	commute in this case. The following
	corollary follows directly from Proposition~\ref{liendo} since all
	tori of dimension $\dim X_{\sigma}$ in $\Aut(X_\sigma)$ are conjugate
	for an affine toric variety $X_\sigma$ (see \cite{MR655409},
	\cite{MR1626643}, and also \cite{MR2000454}).

	\begin{cor} \label{diff-toric} Let $X_\sigma$ be an affine toric
		variety and let $\TT \subset \Aut(X_\sigma)$ be any maximal torus.
		Then all the root subgroups of $\Aut(X_\sigma)$ with respect to
		$\TT$ have different weights. Moreover, if $X$ is not isomorphic to a torus, then it has infinitely many root subgroups.
	\end{cor}

	\subsection{Root subgroups of affine toric surfaces}\label{toricsurf}
	
	In this section we prove our first main result stated in
	Theorem~\ref{toricthm}. The next lemma is known and can, for example,
	be found in \cite[Lemma 10]{kraft2017affine}.
	
	\begin{lem}\label{toruscentralizer}
		Let $X$ be an affine toric variety and let $T\subset\Aut(X)$ be a
		torus of dimension $\dim X$. Then the centralizer of $T$ in
		$\Aut(X)$ equals $T$.
	\end{lem}

	\begin{lem}\label{times}
		Let $X$ be an affine variety and $G,H \subset \Aut(X)$ algebraic
		subgroups such that $G$ normalizes $H$. Then $GH \subset \Aut(X)$ 	is an
		algebraic subgroup.
	\end{lem}
	\begin{proof}
		Let $W_1\subset W_2\subset\cdots$ be  a filtration 
		of the ind-group $\Aut(X)$.
		Since $G$ and $H$
		are algebraic subgroups of $\Aut(X)$, $G \subset W_i$,
		$H \subset W_j$ for some $i$ and $j$. Therefore,
		$GH\subset W_iW_j\subset W_k$ for some $k$.
	\end{proof}

	\begin{lem}\label{obvious}
		Let $G$ be a connected one-dimensional affine algebraic group.  If
		$G$ acts regularly and non-trivially by algebraic group
		automorphisms on a connected one-dimensional affine algebraic group
		$H$, then $G \simeq \GM$ and $H \simeq \GA$.
		
		If $G$ acts regularly and non-trivially by algebraic group
		automorphisms on $\GM\times\GA$ or on $(\GM)^2$, then $G$ is
		isomorphic to $\GM$.
	\end{lem}
	
	\begin{proof}
		Every connected one-dimensional affine algebraic group is isomorphic
		to $\GM$ or to $\GA$. The proof of the first part of the statement    follows now directly from the fact
		that the group of algebraic group automorphisms of $\GM$ is
		isomorphic to $\Z/2\Z$ and the group of algebraic group
		automorphisms of $\GA$ is isomorphic to $\GM$.
		
		Every algebraic group automorphism  of $\GM\times\GA$ is  of the form
		$(x,y)\mapsto (x^{\pm 1}, c(x,x^{-1})y)$, where
		$c(x, x^{-1})\in\C[x,x^{-1}]^*$ is an invertible function on
		$\GM$. In particular, there is no non-trivial $\GA$-action by
		algebraic group automorphisms on $\GM\times\GA$.
		
		Since the group of algebraic group automorphisms of $(\GM)^2$ is
		isomorphic to $\GL_2(\Z)$, no one-parameter group acts on $(\GM)^2$
		by algebraic group automorphisms.
	\end{proof}

	Let $G_{d,e} = \langle g \rangle$ be the cyclic subgroup of order $d$ of
	$\Aut(\A^2)$ that is given by $g\colon (x,y) \mapsto (\xi^e x,\xi y)$,
	where $\xi$ is a $d$-th primitive root of unity, $0\leq e<d$ and
	$(e,d)=1$. Every affine toric surface is either isomorphic to
	$\mathbb{A}_*^1\times \mathbb{A}_*^1$, to
	$\mathbb{A}^1 \times \mathbb{A}^1_*$, or to some
	$X_{d,e} = \mathbb{A}^2/G_{d,e}$.  Furthermore, the toric surface
	$X_{d,e}$ is described in standard correspondence between toric
	varieties and convex polyhedral rational cones by the cone $\sigma$
	spanned by $\rho_1=\beta^*_2$ and $\rho_2=d\beta^*_1-e\beta^*_2$ in
	$N_\RR$, where $\{\beta^*_1,\beta^*_2\}$ is a $\Z$-basis of the
	1-parameter subgroup lattice $N$ of the 2-dimensional torus
	(\cite[Proposition~10.1.3]{MR2810322}). Let $\{\beta_1,\beta_2\}$ be
	the corresponding dual $\Z$-basis of the character lattice $M$ of the
	2-dimensional torus. By \cite[Proposition~10.1.3]{MR2810322} we have
	the following lemma:
	
	\begin{lem}\label{isom-toric-surfaces}
		The toric surface $X_{d,e}$ is isomorphic to $X_{d',e'}$ if and only
		if $d'=d$, and $e=e'$ or $ee'=1\mod d$.
	\end{lem}

	For the proof of Theorem~\ref{toricthm} we need to describe
        the characters of root subgroups of  the automorphism groups of affine toric surfaces
        $X_{d,e}$. By Proposition~\ref{liendo}, root subgroups {of the
          automorphism group of} a toric variety are uniquely
        determined by their weight characters. Let $e'$ and $a$ be the
        unique integers with $0 \leq e'<d$ such that $(e',d)=1$ and
        $ee'=1+ad$. We have the following lemma.
	
	\begin{lem} \label{weight-toric-surface} %
          The weight characters of the root subgroups of
          $\Aut(X_{d,e})$ are:
          \begin{itemize}
          \item with distinguished ray $\rho_1$ the characters
            $\chi^\alpha$ with $\alpha=-\beta_2+l\cdot\beta_1$, for
            all $l\in \Z_{\geq 0}$; and
          \item with distinguished ray $\rho_2$ the characters
            $\chi^\alpha$ with
            $\alpha=(a\beta_1+e'\beta_2)+k\cdot(e\beta_1+d\beta_2)$,
            for all $k \in \Z_{\geq 0}$.
          \end{itemize}
	\end{lem}
	
	\begin{proof}
          By Proposition~\ref{liendo}, weight characters $\chi^\alpha$
          correspond to vectors $\alpha\in M$ with
          $\langle \alpha,\rho_\alpha\rangle=-1$ for some ray
          $\rho_\alpha$ in $\sigma(1)$ and
          $\langle \alpha,\rho\rangle\geq 0$ for all the other
          rays. The ray $\rho_\alpha$ is called the distinguished
          ray. In this case we have only two rays: $\rho_1=\beta^*_2$
          and $\rho_2=d\beta^*_1-e\beta^*_2$. Assume that $\alpha$ is
          of the form $c_1\beta_1+c_2\beta_2$ and let $\rho_1$
          be {a} distinguished ray. Then we have $c_2=-1$ and
          $c_1d-c_2e\geq 0$. This yields the first family in the
          lemma. Let now $\rho_2$ be the distinguished ray. Then we
          have $c_1d-c_2e=-1$ and $c_2\geq 0$. A straightforward
          computation yields the second family in the lemma.
	\end{proof}

	\subsection{Root subgroups of non-toric
		$\GM$-surfaces} \label{sec:non-toric}

              A $\GM$-surface is a surface $S$ together with a given
              faithful regular $\GM$-action on $S$. In
              \cite{MR2020670} a classification of normal affine
              $\GM$-surfaces was given and later this classification
              was generalized in \cite{AlHa06,AHS08} to a
              classification of normal varieties endowed with a torus
              action, the so called \emph{T-varieties}. Recall that
              the \emph{complexity} of a torus action is the
              codimension of a general orbit of the torus $T$. The
              case of complexity zero corresponds to the usual toric
              varieties. In this context, T-varieties of complexity
              one can be seen as a first generalization of toric
              varieties. The case of $\GM$-surfaces is particularly
              well understood.  In this section, we are interested in
              root subgroups of affine $\GM$-surfaces. In
              \cite{MR2196000} a complete classification of such root
              subgroups has been given, which was later generalized to
              arbitrary affine $T$-varieties by the first author in
              \cite{MR2657447}. We recall here the main features of
              the classification that we need in this paper.
	
	\begin{defin}
          The $\GM$-surfaces are classified according to the dynamical
          behavior of the $\GM$-action.  A $\GM$-surface is
          \emph{elliptic} if the $\GM$-action has an attractive fixed
          point, \emph{parabolic} if the $\GM$-action has infinitely
          many fixed points and \emph{hyperbolic} if the $\GM$-action
          has at most finitely many fixed points none of which is
          attractive.
	\end{defin}
	
	In more algebraic terms, a $\GM$-action
        $\alpha\colon\GM\times S\rightarrow S$ on an affine surface
        $S$ gives rise to a $\Z$-grading of the ring of regular
        functions given by
	\[\OO(S)=\bigoplus_{i\in \Z}A_i, \quad\mbox{where}\quad A_i=\left\{f\in\OO(S)\mid \alpha^*(f)=t^i\cdot f \right\} .\]
	
	Elements in $A_i$ are called the \emph{semi-invariants of
          weight $i\in \Z$}. In terms of the $\Z$-grading, a
        $\GM$-surface is hyperbolic if and only if there exist
        non-trivial semi-invariants  with weights  of different signs. In this case
        the generic orbit-closures of $\GM$ are isomorphic to $\A^1_*$. If all
        semi-invariants that are not invariants have  weights of the same sign,
        then the normalizations of the generic orbit-closures {of $\GM$} are
        isomorphic to $\A^1$. If the only invariant functions are the
        constants, we are in the elliptic case. Finally, in the
        parabolic case the ring of invariant functions has
        transcendence degree 1 and so there is a curve of $\GM$-fixed
        points in the surface.
	
 \medskip

 Our next goal is to prove a lemma stating that the centralizer of the acting torus
 $T$ on a non-toric affine $\GM$-surface is a finite extension of
 $T$. For the proof we need to introduce some terminology.  Let $S$ be
 a normal quasi-projective $\GM$-surface. Recall that the {\it
   isotropy group} of a $\GM$-orbit of $S$ is the subgroup of $\GM$ of
 elements fixing the orbit point-wise.  Since the $\GM$-action is
 faithful, all but finitely many $1$-dimensional orbits have trivial
 isotropy. Assume that $S$ admits a categorical quotient
 $\pi\colon S\to C$ to a smooth curve $C$. Then each fiber of the
 quotient map is composed of a union of finitely many $1$-dimensional
 orbit closures since every fixed point of a $\GM$-action is contained
 in a 1-dimensional orbit closure.  We say that a fiber of the
 quotient map $\pi^{-1}(z)$, $z\in C$ is \emph{special} if it contains
 more than one orbit  closure or one orbit {closure} with non trivial isotropy.

 \begin{rem} \label{rem-upgrading} %
   Let $S$ be a normal quasi-projective $\GM$-surface admitting
     a categorical quotient $\pi\colon S\to C$ to a smooth curve
     $C$. It is straightforward to verify that if $C$ admits a {faithful}
     $\GM$-action such that all the points $z$ lying below a special
     orbit $\pi^{-1}(z)$ are fixed, then this $\GM$-action lifts
     to $S$ and so $S$ is toric. It turns out that the converse is
     also true by virtue of \cite[Section~11]{AlHa06}.
 \end{rem}
	
\begin{lem}\label{centralizer}
  Let $S$ be a non-toric normal affine $\GM$-surface and denote by
  $T\subset\Aut(S)$ the subgroup isomorphic to $\GM$ induced by the
  $\GM$-action. Then the centralizer $Z$ of $T$ in $\Aut(S)$ contains
  $T$ as a normal subgroup of finite index.
\end{lem}
	
\begin{proof}
  We first treat the parabolic and hyperbolic cases.  In { these cases}
     {$C = S/\!\!/T$ }  is a smooth affine
    curve, because $S$ is normal and $\GM$ is reductive. In both
    cases, being non-toric implies by Remark~\ref{rem-upgrading} that
    $C$ is either non-rational, or $C$ is a proper subset of $\A^1_*$,
    or $C=\A^1_*$ and the quotient map has at least one special fiber,
    or $C=\A^1$ and the quotient map has at least two special fibers.
		
  {In the elliptic case, it is well known that $S$ is smooth if
    and only if $S$ is the affine space (see
    \cite[Theorem~2.5]{MR644276}) and this is excluded since it is
    toric. It follows that every automorphism of $S$ must fix the only
    fixed point $\bar{0}$ that is also the only singular point of
    $S$. Hence, every automorphism induces an automorphism of
    $S_0=S\setminus\{\bar{0}\}$ . In this case, there is also a
    categorical quotient $\pi\colon S_0\to C$ to a smooth projective
    curve $C$. In this case, being non-toric by
    Remark~\ref{rem-upgrading} implies that $C$ is non-rational or
    $C=\mathbb{P}^1$ and the quotient map $\pi$ has at least three
    special fibers.}
		
  In all the three cases, the centralizer $Z$ of $T$ induces an action
  on the categorical quotient $C$. Let $G$ be the image of the
  centralizer action in $\Aut(C)$. We will show that $G$ is finite. If
  $C$ is a proper subset of $\A^1_*$, or $C$ is non-rational, or $C$
  is not an elliptic curve, then $\Aut(C)$ is already finite.  Assume
  first that $C$ is rational. Let $F\subset C$ be the image of the
  special fibers of $\pi$ in $C$. This set $F$ must be preserved by
  $G$.  Hence, $G$ is a finite group since it fixes a finite set of
  cardinality at least one if $C=\A^1_*$, cardinality at least two if
  $C=\A^1$ and cardinality at least three if $C=\mathbb{P}^1$. Assume
  now that $C$ is an elliptic curve. If $F$ is non empty, then it is
  preserved by $G$. Again it follows that $G$ is finite since the
  automorphism group of an elliptic curve fixing a point is
  finite. 
  
  The only remaining case is the case where there are no
  special fibers and the quotient curve is elliptic.
  This case was already studied by Demazure in
  \cite[Theorem~3.5]{Dem88}. The surface $S$ is isomorphic to the ring
  of sections of an ample divisor $D$ on $C$, i.e.,
  \[S=\spec \bigoplus_{i\in \Z_{>0}} A_i,\quad\mbox{where}\quad
    A_i=H^0(C,\OO_C(iD))\,.\] %
  Now, by \cite[Corollary~8.9]{AlHa06}, elements in the image
  $G\subset \Aut(C)$ of the centralizer $Z$ are in one-to-one
  correspondence with automorphisms $\phi\colon C\to C$ such that $D$
  is linearly equivalent to $\phi^*(D)$. By definition, $D$ is
  linearly equivalent to $\phi^*(D)$ if and only of $D-\phi^*(D)$ is
  principal.  Since the automorphism group of the elliptic curve consists of translations up to finite index, so we only need to consider  translations in order to prove that $G$ is finite. Let $\phi_p$ be the translation by an element $p$ in the
  elliptic curve $C$. We claim that $D-\varphi_p^*(D)$ is principal if
  and only of $p$ is the identity $e\in C$. This  {claim implies} that $G$ is
  finite since it is contained in the automorphisms of $C$ that fix
  $e$.   {To prove the claim we first note that}  the isomorphism of an elliptic curve
  $C\to\operatorname{Pic}^0(C)$ into the connected component of its
  Picard scheme is given by $p\mapsto [p]-[e]$, where $e\in C$ is the
  identity. This yields that $[p+q]$ is mapped to $[p]+[q]-[e]$ and
  now a straightforward computation yields $D-\phi_p^*(D)$ is linearly
  equivalent to $\deg(D)\cdot([e]-[p])$. Now, this last divisor is
  principal if and only if $p=e$. This proves the claim.

  To conclude the proof, remark that the generic fibers of the
  categorical quotient are always affine rational curves endowed with
  a $T$-action. This yields the following exact sequence
  $ 0\rightarrow H\rightarrow Z\rightarrow G\rightarrow 0$, where the
  elements in $H\subset\Aut(S)$ preserve the fibration
  $S\rightarrow S /\!\!/ T$ fiberwise. Since $H$ commutes with $T$ and the generic fibers are rational curves, we
  obtain that $H=T$.
\end{proof}

        In the sequel, we will require the following corollary of the
        above lemma.
        
	\begin{cor}\label{toricma}
          If $\GA \times \GM$ acts faithfully on a normal affine
          surface $X$, then $X$ is toric.
	\end{cor}

	In this section we are only interested in root subgroups of non-toric
	$\GM$-surfaces. This considerably restricts the root subgroups we
	encounter. In the elliptic case, by \cite[Theorem~3.3]{MR2196000}
	only toric surfaces admit root subgroups in their automorphism
	group. In the non-toric parabolic case, by \cite[Theorems~3.12 and
	3.16]{MR2196000} the ring of invariants of any root subgroup equals
	the ring of $\GM$-invariants. In the language of \cite{MR2196000} and
	\cite{MR2657447} these root subgroups are called \emph{fiber type}.
	Finally, in the hyperbolic case, by \cite[Lemma~3.20]{MR2196000} the
	ring of invariants of any root subgroup intersected with the ring of
	$\GM$-invariants is only the constants. In the language of
	\cite{MR2196000} and \cite{MR2657447} these root subgroups are called
	\emph{horizontal type}.

\begin{lem} \label{different-weight} Let $S$ be a non-toric
  $\GM$-surface. Then
  \begin{enumerate}
  \item The surface $S$ admits root subgroups of different weights if
    and only if $S$ is hyperbolic.
  \item The surface $S$ does not admit root subgroups of positive and
    negative weights with the same generic orbits.
  \item If $S$ is hyperbolic, then all root subgroups have different
    weights.
  \end{enumerate}

\end{lem}
	
\begin{proof}
  Recall that in the non-toric case, only parabolic and hyperbolic
  $\GM$-surfaces admit root subgroups. By
  \cite[Theorem~3.12]{MR2196000} all root subgroups of a parabolic
  surface have the same weight equal to $1$ or $-1$. Furthermore, in
  the hyperbolic case, given any root subgroup $U$ with corresponding
  locally nilpotent derivation $\delta$ we obtain another root
  subgroup $U$ with different weight by considering the locally
  nilpotent derivation $f\cdot\delta$ where $f$ is a non-constant
  element in $\ker\delta$ that is also semi-invariant, but can not be
  $\GM$-invariant, since the action is hyperbolic. This proves the
  first statement.

  We now show (b). Assume now that $S$ admits two root subgroups $U$
  and $U'$ with the same generic orbits and assume also that the
  weight of $U$ is positive and the weight of $U'$ is negative. By
  part (a), we have that $S$ is hyperbolic. Moreover, the conditions
  to have the same generic orbits implies that the rings of invariants
  $\OO(S)^{U}$ and $\OO(S)^{U'}$ coincide. Furthermore, since $U$ has
  positive weight, there exists a non-trivial $U$-invariant $f$ that
  is $\GM$-semi-invariant of positive degree $n$. Similarly, since
  $U'$ has negative weight, there exists a non-trivial $U$-invariant
  $f'$ that is $\GM$-semi-invariant of negative degree $-n'$. We thus
  obtain that $f^{n'}(f')^{n}$ is a non-trivial invariant by $\GM$ and
  $U$. Since $S$ is hyperbolic, we conclude that $f^{n'}(f')^{n}$ is a
  non-zero constant and so $f$ is a unit that is $\GM$-semi-invariant
  of non-zero degree. Part (b) now follows from
  \cite[Corollary~3.27~(ii)]{MR2196000}

  The third statement follows directly from the main classification
  theorem for hyperbolic surfaces \cite[Theorem~3.22~(iii)]{MR2196000}
  since the root subgroup is uniquely determined by  {its} weight.
\end{proof}
	
\section{Automorphisms of automorphism groups of toric
  surfaces}\label{autaut}
	
	In this section we look at group automorphisms of $\Aut(S)$,
        where $S$ is a toric surface. Recall that all toric surfaces
        are defined over the field $\QQ$ of rational numbers
        (\cite{MR0284446}). Hence, all field automorphisms $\tau$
        of $\C$ induce a base-change and hence a permutation of the
        $\C$-points of $S$. By conjugating elements in $\Aut(S)$ by
        this permutation, we obtain a group automorphism of $\Aut(S)$,
        which, by abuse of notation, we denote by $\tau$ as well. If
        $G\subset\Aut(S)$ is an algebraic subgroup, then $\tau(G)$ is,
        by base-change, again an algebraic subgroup of $\Aut(S)$. Note
        that if $\tau(G)=G$ and if we chose any coordinates of $G$,
        then the restriction of $\tau$ to $G$ is just given by
        applying $\tau$ to the coordinates of $G$.
	
	The goal of this section is to prove the following theorem:
	
	\begin{prop}\label{groupautomorphismtoric}
		Let $S$ be a toric surface and $\varphi$ a group automorphism of $\Aut(S)$. Let $G\subset\Aut(S)$ be an
		algebraic subgroup.
		Then $\varphi(G)$ is again an algebraic group that is isomorphic to $G$.
		
		Moreover, if $S$ is not isomorphic to $\A^1_*\times\A^1_*$, then there exists a field-automorphism $\tau$ of $\C$ such that the restriction of $\tau\circ\varphi$ to $G$ is an algebraic morphism.
	\end{prop}

	In \cite{MR2278755} it is shown that all group automorphisms of
	$\Aut(\A^2)$ are inner up to automorphisms of the base-field
	$\C$. Note that this result and our Proposition \ref{A2} will imply directly the following observation (which will not be used in the sequel):
	
	\begin{cor}\label{autcor}
		Let $S$ be an affine surface and
		$\varphi\colon\Aut(\A^2) \to \Aut(S)$ a group isomorphism. Then there exists an isomorphism
		$f\colon \A^2\to S$ and a field automorphism $\tau$ of $\C$ such that $\varphi(g)=f\tau(g)f^{-1}$ for all $g\in \Aut(\A^2)$.
	\end{cor}
	
	This corollary can be seen as an algebraic analogue to a result of
	Filipkiewicz (\cite{MR693972}) about isomorphisms between groups of
	diffeomorphisms of manifolds without boundary. 
	
	\begin{lem}\label{autosemidirect}
		Let $\chi\colon \GM\to\GM$ be a non-trivial character and $\varphi\colon \GM\ltimes_\chi\GA\to\GM\ltimes_\chi\GA$ be a group automorphism. Then there exists a field automorphism $\tau\colon\C\to\C$ such that $\varphi\circ\tau$  is inner.
	\end{lem}
	
	\begin{proof}
		Choose coordinates $x$ and $y$ of $\GM$ and $\GA$ respectively, so every element in $\GM\ltimes_\chi\GA$ is of the form $(x,y)$ and composition is given by $(x,y)(x',y')=(xx', \chi(x)y'+y)$. Every non-trivial normal subgroup of $\GM\ltimes_\chi\GA$ contains $\GA$. Indeed, let $N\subset\GM\ltimes_\chi\GA$ be a non-trivial normal subgroup and let $(a,b)\in N\setminus \id$. Then for every $(c,d)\in \GM\ltimes_\chi\GA$ the element 
		\[
		(a,b)(c,d)(a,b)^{-1}(c,d)^{-1}=(1,f)
		\]
		is contained in $N$ with $f\neq 0$ for general $(a,b), (c,d)$. Now $(c,0)(1,f)(c^{-1},0)=(1,\chi(c)f)$ and hence $\GA\subset N$. With this characterization of $\GA$, it follows that $\varphi(\GA)=\GA$. 
		
		Now assume that the restriction of $\varphi$ to $\GM$ is given by $(a,0)\mapsto (\varphi_1(a),\varphi_2(a))$, where $\varphi_1\colon\GM\to\GM$ is a homomorphism of groups and $\varphi_2\colon\GM\to\GM$ a map such that 
		$\varphi_2(ab)=\chi(\varphi_1(a))\varphi_2(b)+\varphi_2(a)$.
		Since $ab=ba$ we obtain  
		$\chi(\varphi_1(a))\varphi_2(b)+\varphi_2(a) = \chi(\varphi_1(b))\varphi_2(a)+\varphi_2(b)$
		and hence
		\begin{equation}\label{notdep}
			\frac{\varphi_2(a)}{\chi(\varphi_1(a))-1}=\frac{\varphi_2(b)}{\chi(\varphi_1(b))-1}
		\end{equation}
		if $\varphi_1(a)$ and  $\varphi_1(b)$ are not contained in the kernel of $\chi$. By the observation above, the elements in $\GM\ltimes_\chi\GA$ of the form $(1,c)$ are contained in the image of $\GA$ and therefore not in the image of $\GM$. Since $\varphi$ is an isomorphism, we hence get that $\varphi_1(a)=1$ if and only if $a=1$. 
		
		One calculates for $a\neq 1$
		\[
		\left(1,\frac{\varphi_2(a)}{\chi(\varphi_1(a))-1}\right)\left(\varphi_1(a),\varphi_2(a)\right)\left(1,-\frac{\varphi_2(a)}{\chi(\varphi_1(a))-1}\right)=\left(\varphi_1(a),0\right).
		\]
		By Equation (\ref{notdep}), the element $\left(1,\frac{\varphi_2(a)}{\chi(\varphi_1(a))-1}\right)$ does not depend on $a$. So up to conjugation by an element from $\GA$ we may assume that $\varphi(\GM)=\GM$.
		Conjugating with a suitable element of $\GM$ we  can moreover assume that 
		$\varphi(1,1)=(1,1)$.
		
		Let $\rho\colon\C\to\C$ be the map such that $\varphi(1,c)=(1,\rho(c))$ and $\delta\colon\C\to\C$ the map such that $\delta(0)=0$ and $\varphi(c,0)=(\delta(c), 0)$ for $c\in\C^*$. This yields $\rho(a+b)=\rho(a)+\rho(b)$ and $\delta(cd)=\delta(c)\delta(d)$, since $\varphi$ is a homomorphism of groups. We calculate
		\[
		(1,\chi(\delta(c)))=(\delta(c),0)(1,1)(\delta(c)^{-1},0)=\varphi(c,0)\varphi(1,1)\varphi(c^{-1},0)=\varphi(1,\chi(c))=(1,\rho(\chi(c)))).
		\]
		Since $\delta(\chi(c))=\chi(\delta(c))$ for all $c\in\C^*$ and since $\chi$ is non-trivial and therefore surjective, we obtain that $\delta=\rho$.
		For $a,b,c\in\GA$ we have
		\begin{align*}
			(1,\chi(c)(a+b))&=(c,0)(1,a+b)(c^{-1},0)\\
			&=(c,0)(1,a)(c^{-1},0)(c,0)(1,b)(c^{-1},0)\\
			&=(1,\chi(c)a)(1,\chi(c)b)
		\end{align*}
		and therefore, again by surjectivity of $\chi$, we obtain $\rho(c(a+b))=\rho(ca)+\rho(cb)$. So $\rho$ is a field automorphism. 
		Hence the restriction of $\rho^{-1}(\varphi)$ to $\GM$ and to $\GA$ is the identity. We conclude that $\rho^{-1}(\varphi)=\id$. This implies the claim.
	\end{proof}
	
	\begin{lem}\label{root}
		Let $S_1$ and $S_2$ be two affine surfaces endowed with faithful regular actions of positive
		dimensional tori $T_1$ and $T_2$ respectively such that $T_1\subset\Aut(S_1)$ and $T_2\subset\Aut(S_2)$ are maximal tori. Assume that
		there is an isomorphism $\phi\colon \Aut(S_1) \to \Aut(S_2)$ such
		that $\phi(T_1) = T_2$.  Then any root subgroup of $\Aut(S_1)$ with
		respect to $T_1$ is sent by $\phi$ to a root subgroup of $\Aut(S_2)$
		with respect to $T_2$.
	\end{lem}
	
	\begin{proof}
		If $\Aut(S_1)$ contains no root subgroup, the lemma is trivially true. Otherwise, let $U$ be a root subgroup of $\Aut(S_1)$ with respect to $T_1$, then
		$T_1$ acts on $U$ with two orbits, namely $\{\id\}$ and $U\setminus\id$. Hence, $\phi(U)$ is a
		group normalized by $T_2 = \phi(T_1)$ which also acts on
		$\phi(U)$ with two orbits, which are contained in a filter set of the ind-group $\Aut(S_2)$. Since any orbit of the algebraic
		group action is open in its closure, we obtain, by considering the dimension of the kernel of the actions, that
		$\phi(U)\setminus\{\id\}$ is an irreducible quasi-affine curve. This implies that
		$\overline{\phi(U)}=\phi(U)$ is a connected one-dimensional algebraic
		group. Since
		$\phi(U)$ contains at most one root of unity,
		$\overline{\phi(U)}$ is isomorphic to $\GA$ and
		normalized by $\phi(T_1)=T_2$. 
	\end{proof}
	
		\begin{prop}\label{tor}
		Let $S_1$ be an affine toric surface that is not isomorphic to $\A_*^1\times\A_*^1$ and let $S_2$ be an affine surface. Let $T\subset\Aut(S_1)$ be a maximal torus. Assume that there is an isomorphism of groups $\phi\colon\Aut(S_1) \to \Aut(S_2)$. Then we have the following:
		\begin{itemize}
			\item The image $\phi(T)$ is a maximal torus of rank 2 in $\Aut(S_2)$.
			\item  The normalization of  $S_2$ is an affine toric surface that is not isomorphic to $\A_*^1\times\A_*^1$.
			\item If, moreover, $S_2$ is normal, then it is isomorphic to $\mathbb{A}^1_* \times \mathbb{A}^1$ if and only if $S_1$ is isomorphic to $\mathbb{A}^1_* \times \mathbb{A}^1$.
		\end{itemize}

	\end{prop}
	
	\begin{proof}
		
		Since $S_1$ is not isomorphic to $\A_*^1\times\A_*^1$, $\Aut(S_1)$ contains two root subgroups $U_1$ and $U_2$ with
		respect to the maximal torus $T$ with different non-trivial weights $\chi_1$ and $\chi_2$, by Proposition~\ref{liendo}. Since $\chi_1$ and $\chi_2$ are different, the restriction of $\chi_1$ to $\ker(\chi_2)$ and the restriction of $\chi_2$ to $\ker(\chi_1)$ are both surjective homomorphisms to $\GM$. Let $t_1\in \ker(\chi_1)$ be such that $\chi_2(t_1)=2$ and $t_2\in\ker(\chi_2)$ such that $\chi_1(t_2)=2$. Let $u_1\in U_1$ and $u_2\in U_2$ be two non-identity elements. Then $t_1$ commutes with $u_1$ and $t_2$ commutes with $u_2$, but $t_1u_2t_1^{-1}=u_2^2$ and $t_2u_1t_2^{-1}=u_2^2$. We now observe that the two algebraic groups $\overline{\langle \phi(t_1) \rangle}^\circ$ and $\overline{\langle \phi(t_2) \rangle}^\circ$ (by Proposition~\ref{alg}, the elements $\phi(t_1)$ and $\phi(t_2)$ are algebraic) commute with each other and do not coincide. For $i \in \{1,2\}$ the group $\overline{\langle \phi(u_i) \rangle}^\circ$ is isomorphic to $\GM^n\times\GA^m$, where $m\in\{0,1\}$ and $n\leq 2$. By Lemma~\ref{obvious}, we have that $\overline{\langle \phi(u_i) \rangle}^\circ$ is not isomorphic to  $(\GM)^n$.   Moreover, $\overline{\langle \phi(u_i) \rangle}^\circ$ is not isomorphic to  $(\GM)^2\times\GA$, since on a toric variety a maximal torus coincides with its centralizer (Proposition~\ref{toruscentralizer}). By applying again Lemma~\ref{obvious}, we obtain that $\overline{\langle \phi(t_1) \rangle}^\circ$ and $\overline{\langle \phi(t_2)\rangle}^\circ$ each contain a subgroup isomorphic to $\GM$. By Proposition~\ref{toruscentralizer}, $\phi(T)=(\GM)^2$, which proves the first statement. 
		
		By definition, the normalization of $S_2$ is a toric surface. Moreover, $\A_*^1\times\A_*^1$ is the only toric surface such that its automorphism group  normalizes its maximal torus. Since $S_1$ is not isomorphic to $\A_*^1\times\A_*^1$ by assumption, $\Aut(S_1)$ contains an element that does not normalize the maximal tours. By the first statement of the proposition, $\Aut(S_2)$ contains an element that does not normalize the maximal torus. Since $\Aut(S_2)$ is contained as a subgroup in the automorphism group of the normalization of $S_2$, the normalization of $S_2$ is therefore not isomorphic to $\A_*^1\times\A_*^1$. 
		
		By Lemma~\ref{root}, the isomorphism $\phi$ maps root subgroups of $S_1$ with respect to $T$ to root subgroups of $S_2$ with respect to $\phi(T)$. We observe that $\mathbb{A}^1_* \times \mathbb{A}^1$ is the only
		affine toric surface such that all its root subgroups commute, and thus obtain the last statement of the lemma.
	\end{proof}
	
	\begin{lem}\label{autauttorus}
		Let $S$ be a toric surface not isomorphic to the algebraic torus and let $\varphi$ be a group automorphism of  $\Aut(S)$. Then for any maximal torus $T\subset \Aut(S)$ and any root subgroup $U\subset\Aut(S)$ the images $\varphi(T)$ and $\varphi(U)$ are algebraic groups. Moreover,  there exists a field automorphism $\tau$ of $\C$ such that the restrictions of $\varphi\circ\tau$ to $T$ and $U$ are isomorphisms of algebraic groups.
	\end{lem}
	
	\begin{proof}
		Since all maximal tori are conjugate in $\Aut(S)$, it is enough to show the statement for a fixed maximal torus $T\subset\Aut(S)$.
		Let $U_1, U_2, U_3\subset\Aut(S)$ be root subgroups with respect to  $T\subset \Aut(S)$ with different weights $\chi_1$, $\chi_2$ and $\chi_3$. Define the three 1-dimensional subtori $T_1=\ker(\chi_1)$, $T_2=\ker(\chi_2)$ and $T_3=\ker(\chi_3)$. Since the weights are different, the intersections $T_i\cap T_j$ are finite for $i\neq j$. 
		By Proposition \ref{tor}, $\varphi(T)$ is a maximal torus of $\Aut(S)$.
		By Lemma~\ref{root}, $\varphi(U_1)$, $\varphi(U_2)$ and $\varphi(U_3)$ are again root subgroups with respect to the maximal torus $\varphi(T)$ and since the $\varphi(T_i)$ commute with $\varphi(U_i)$, we obtain that the $\varphi(T_i)$ are closed one-dimensional subtori of $\varphi(T)$. By Lemma~\ref{autosemidirect}, there exists a field automorphism $\tau_1$ of $\C$ such that the restriction of $\varphi\circ \tau_1$ to $T_1\ltimes_{\chi_3} U_3$ is an algebraic isomorphism. Similarly, there exists a field automorphism $\tau_2$ of $\C$ such that the restriction of $\varphi\circ \tau_{{2}}$ to $T_2\ltimes_{\chi_3} U_3$ is an algebraic isomorphism. Since the restrictions of both, $\varphi\circ\tau_1$ and $\varphi\circ\tau_2$ to $U_3$ are algebraic, we obtain that $\tau_1=\tau_2$ using that the identity is the only field automorphism inducing an algebraic homomorphism.
		It follows that the restriction of $\varphi \circ \tau_1$ to $T_1$ and {to} $T_2$ is algebraic. Since every element in $T$ is  a product of an element of $T_1$ with an element of $T_2$, we obtain that the restriction of $\varphi\circ\tau_1$ to $T$. Let $U$ be any root subgroup with respect to $T$. Since $T$ acts transitively on $U\setminus\{\id\}$, this implies that the restriction of $\varphi\circ\tau_1$ to any root subgroup $U$ is an isomorphism of algebraic groups to the algebraic group $\varphi\circ\tau_1(U)$.
	\end{proof}
	
	\begin{lem}\label{autdeg}
		Let $\varphi$ be an automorphism of $\Aut(\A^1_*\times\A^1_*)$. Then $\varphi$ maps algebraic groups to algebraic groups.
	\end{lem}
	
	\begin{proof}
		It is enough to show the statement for connected groups.
		Recall that $\Aut(\A^1_*\times\A^1_*)\simeq T\rtimes \GL_2(\Z)$, where we identify $\GL_2(\Z)$ with the group of monomial maps, i.e. transformations of the form $(x,y)\mapsto (x^ay^b,x^cy^d)$, for a matrix $\left[ {\begin{array}{cc}
			a & b \\
			c & d \\
			\end{array} } \right] \in \GL_2(\mathbb{Z})$. First we note that $\varphi(T)=T$. Indeed, consider the projection $\pi\colon \Aut(\A^1_*\times\A^1_*)\to\GL_2(\Z)$. Since  all elements in $\varphi(T)$ are divisible and since $\varphi(T)$ is a maximal commutative subgroup, we obtain that $\varphi(T)=\ker(\pi)=T$. 
	
		Set
		\[
		M:=
		\left[ {\begin{array}{cc}
			1 & 1 \\
			0 & 1 \\
			\end{array} } \right] \in \GL_2(\mathbb{Z}).
		\]
		Consider the closed subgroup $T_0 = \{ t \in T \mid tM = Mt  \}\subset T$. Direct calculations show that
		$T_0 = \{ (t,1) \mid t \in \mathbb{C}^*   \} \simeq \GM$. Hence,
		\[ \varphi(T_0) = \{ \tilde{t} \in \varphi(T) \mid \tilde{t} \varphi(M) = \varphi(M) \tilde{t}  \} \] is a closed one-dimensional subgroup of $\varphi(T)=T$ that is as an abstract group isomorphic to $\GM$ and hence connected.  
		Thus, the image of $T_0$ under the isomorphism $\varphi$ is a closed connected subgroup of $\Aut(\A^1_*\times\A^1_*)$. 
		
		Let now $T'\subset T$ be any closed connected one-dimensional subgroup. Then $T'$ is of the form $\{(t_1, t_2)\in T\mid t_1^at_2^b=1\}$ for some integers $a,b$. Moreover, since $T'$ is connected, we have $\gcd(a,b)=1$. Let $c,d\in\Z$ be such that $ad-bc=1$ and define 
		\[
		N:=
		\left[ {\begin{array}{cc}
			a & b \\
			c & d \\
			\end{array} } \right] \in \GL_2(\mathbb{Z}).
		\]
		One calculates that $NT'N^{-1}=T_0$. Therefore, the image $\varphi(T')$ is conjugate to a connected closed one-dimensional subgroup of $T$ and is therefore itself closed, connected, and one-dimensional. For the zero-dimensional and the two-dimensional closed connected subgroup of $\Aut(\A^1_*\times\A^1_*)$ the statement is clear.
	\end{proof}
	
	\begin{lem}\label{immulemma}
		Let $X,Y,Z$ be affine varieties, $f\colon X\to Y$ an abstract map and $g\colon Z\to X$ a surjective morphism such that $f\circ g\colon Z\to Y$  is a morphism. If $Y$ is normal, then $f$ is a morphism.
	\end{lem}
	
	\begin{proof}
		By Zariski's main theorem, it is enough to show that the graph $\Gamma_f\subset X\times Y$ is closed. Since $h=g\circ f$, the graph $\Gamma_h\subset Z\times Y$ is closed. Consider the surjective morphism $\rho\colon Z\times Y\to X\times Y$ mapping $(z, y)$ to $(g(z), y)$. Note that $\rho(\Gamma_h)=\Gamma_f$. Consider the Zariski-closure $\overline{\Gamma_f}$ of $\Gamma_f$. The projection $\theta\colon \overline{\Gamma_f}\to X$ is surjective and $\overline{\Gamma_f}$ contains an open dense set $U$ such that the restriction of $\theta$ to $U$ is injective. This implies, again by Zariski's main theorem, that $\theta$ is surjective and therefore, since $\overline{\Gamma_f}$ and $X$ are affine, that $\theta$ is an isomorphism and hence that $\Gamma_f=\overline{\Gamma_f}$. In particular, $\Gamma_f$ is closed.
	\end{proof}

	\begin{lem}\label{oneparalg}
		Let $G$  and $H$ be linear algebraic groups and let $V_1,\dots V_n\subset G$ be subgroups such that $G=V_1\cdot\dots\cdot V_n$.  If $\varphi\colon G\to H$ is a homomorphism of groups such that the restriction of $\varphi$ to $V_i$ is an algebraic homomorphism for all $i$, then $\varphi$ is an algebraic homomorphism.
	\end{lem}
	
	\begin{proof}
		We have to show that $\varphi$ is a morphism. Consider the morphism $h\colon V_1\times\dots\times V_n\to H$ defined by $(v_1,\dots, v_n)\mapsto \varphi(v_1)\varphi(v_2)\dots\varphi(v_n)$ and the surjective morphism $g\colon  V_1\times\dots\times V_n\to G$ defined by $(v_1,\dots, v_n)\mapsto v_1v_2\dots v_n$. Then $h=\varphi\circ g$. Since $h$ and $g$ are morphisms and all the varieties involved are normal, we obtain by Lenmma~\ref{immulemma} that $\varphi$ is a morphism. 
	\end{proof}

	Let $e$ and $d$ be integers such that  $1 \le e < d$,
	$(e,d) = 1$ and let $X_{d,e}=\A^2/G_{d,e}$ be the corresponding toric surface (see Section~\ref{toricsurf}). Denote by $\mathcal{N}_{d,e}\subset \Aut(\A^2)$ the normalizer of $G_{d,e}$ and define \[
	\Jonq^+(\A^2)=\{(\alpha x +p(y), \beta x + \gamma)\mid \alpha, \beta\in\C^*, p(y)\in\C[y], \gamma \in \mathbb{C} \},\]\[
	\Jonq^-(\A^2)=\{(\alpha x + \gamma, \beta x +p(x))\mid \alpha, \beta\in\C^*, p(x)\in\C[x], \gamma \in \mathbb{C}\}.\] 
	Set $\mathcal{N}^\pm_{d,e}=\mathcal{N}_{d,e}\cap \Jonq^\pm(\A^2)$.
	Define $N_{d,e}$ as the normalizer of $G_{d,e}$ in $\GL_2(\C)\subset\Aut(\A^2)$ and set $N_{d,e}^\pm=N_{d,e}\cap\Jonq^\pm$. Note that $N_{d,e}^\pm\subset\mathcal{N}^\pm_{d,e}$ and by \cite[Section 4.1]{MR3089030} 
	\[  N_{d,e} =  
	\begin{cases}
	\GL_2(\mathbb{C}),               & \text{ if }  e = 1,\\    
	N(T) = \langle T,\tau \rangle    & \text{ if } e > 1 \text{ and  } e^2  \equiv 1  \text{ mod } d    \\
	T  & \text{otherwise},
	\end{cases} \]
	where  $\tau : (x, y) \mapsto (y, x)$ is a twist and $T$ is the maximal torus in $\GL(2, \mathbb{C})$ consisting of the diagonal matrices.  Denote by $B^{\pm}$  the Borel subgroup of all upper (lower,
	respectively) triangular matrices in $\GL(2,\mathbb{C})$. Note that
	\[  N_{d,e}^{\pm}  =
	\begin{cases}
	B^{\pm},               & \text{ if }  e = 1,\\  
	T  & \text{otherwise}.
	\end{cases} \]
	Then, by \cite[Theorem 4.2]{MR3089030} we have the following amalgamated product structures
	\[
	\Aut(X_{d,e})\simeq\mathcal{N}^+_{d,e}/G_{d,e}*_{T/G_{d,e}}\mathcal{N}^-_{d,e}/G_{d,e}, \text{ if $e^2\not\equiv 1\mod d$},\]\[
	\Aut(X_{d,e})\simeq\mathcal{N}^+_{d,e}/G_{d,e}*_{N^+_{d,e}/G_{d,e}}N_{d,e}/G_{d,e}, \text{ if $e^2\equiv 1\mod d$},
	\]
	where the actions of  $\mathcal{N}^\pm_{d,e}/G_{d,e}$, $T/G_{d,e}$ and $N_{d,e}/G_{d,e}$ on $X_{d,e}$ are the actions induced by the actions of $\mathcal{N}^\pm_{d,e}$, $T$ and $N_{d,e}$ on $\A^2$ and then passing to the quotient $X=\A^2/G_{d,e}$. In particular, an algebraic subgroup of $\mathcal{N}^\pm_{d,e}$, $T$, or $N_{d,e}$ descends to an algebraic subgroup of $\mathcal{N}^\pm_{d,e}/G_{d,e}$, $T/G_{d,e}$, or $N_{d,e}/G_{d,e}$.

	Another result, which we need, is that an algebraic subgroup $G\subset\Aut(X_{d,e})$ is always conjugate to a subgroup of one of the factors of the above amalgamated product structures (\cite[Theorem 4.15]{MR3089030}).

	\begin{lem}\label{unipotentsubgroups}
	Any unipotent subgroup of $\Jonq^{+}(\mathbb{A}^2)$ is either commutative or
    the semidirect product $U \ltimes V$, where
    \[ U = \{ (x + \gamma,y) \mid \gamma \in \mathbb{C} \}, \; \; V= \{ (x,y +p(x)) \mid p\in\mathbb{C}[x]_{\le k}:= {\mathbb{C}\oplus \mathbb{C}x \oplus \dots \oplus \mathbb{C}x^k}  \}, \]
    for some $k\geq 0$.
	\end{lem}
    \begin{proof}
      Assume that $G \subset \Jonq^{+}(\mathbb{A}^2)$ is a
      non-commutative unipotent subgroup.  Then $G$ has nilpotent
      length $2$, more precisely, $G = G_1 \ltimes G_2$, where
      $G_1 = U$ and $G_2 = \{ (x,y + p )\mid p\in L \}$ for some
      finite dimensional vector subspace $L \subset \C[x]$. Since
      $G_1$ acts on $G_2$ by conjugation it follows that $\Lie G_1$
      acts on $\Lie G_2$ by the adjoint representation. The Lie
      algebra $\Lie G$ can be naturally embedded into the Lie algebra
      of derivations $\Der(\mathcal{O}(\A^2))$ and we identify
      $\Lie G$ with its image in $\Der(\mathcal{O}(\A^2))$. Hence,
      $\Lie G_1 = \mathbb{C} \frac{\partial}{\partial x}$ and
      $\Lie G_2=L \frac{\partial}{\partial y}$. The action of
      $\Lie G_1$ on $\Lie G_2$ induces the action of
      $ \frac{\partial}{\partial x}$ on $L$.  Let $f \in L$ be a
      polynomial of maximal degree. Then
      $(\frac{\partial}{\partial x})^{\deg f}f$ is a non-zero constant
      and hence, $L$ contains $\mathbb{C}$. Further,
      $(\frac{\partial}{\partial x})^{\deg f - 1}f$ is a polynomial of
      degree $1$ and we have that
      $\mathbb{C} \oplus \mathbb{C}x \subset L$. Following this
      procedure we conclude that $L = \mathbb{C}[x]_{\le k}$.
    \end{proof}  	 

	\begin{lem}\label{basechangelemma}
		Let $S$ be a toric surface and $G\subset\Aut(S)$ a connected algebraic subgroup. Let $\tau$ be a field automorphism of $\C$ and let $G'$ be the algebraic group obtained by base-changing the field $\C$ by $\tau$ over $\Q$. Then $G'$ is isomorphic to $G$ as an algebraic group over $\C$.
	\end{lem}
	
	\begin{proof}
          It is enough to show that all connected algebraic subgroups
          of $\Aut(S)$ are defined over the field $\Q$. We do this by
          exhaustion of cases.  First assume that $S\simeq X_{d,e}$
          for some $e$ and $d>1$.  By the above stated results of
          Arzhantsev and Zaidenberg \cite{MR3089030}, a connected
          algebraic subgroup of $\Aut(X_{d,e})$ is either conjugate to
          a subgroup of a quotient of the subgroups of
          $\Jonq^{\pm} (\mathbb{A}^2 )$ or a quotient of a subgroup of
          $N_{d,e}$.  The connected quotients of subgroups of
          $N_{d,e}$ are either $\GL(2,\mathbb{C})$,
          $\SL(2,\mathbb{C})$, $\PGL(2,\mathbb{C})$,
          $\PGL(2,\mathbb{C}) \times \mathbb{G}_m$, commutative
          unipotent subgroups, tori of rank at most 2, or a semidirect
          product of a torus and a commutative unipotent subgroup. On
          the other hand, connected algebraic subgroups of
          $\Aut(\mathbb{A}^2)$ are either isomorphic to one of those
          which are already listed above, to $\Aff(2,\mathbb{C})$, to
          $\SAff(2,\mathbb{C})$ or to a non-commutative unipotent
          group as described in
          Lemma~\ref{unipotentsubgroups}. Furthermore, connected
          algebraic subgroups of $\Aut(S\simeq \A^1_*\times\A^1_*)$
          are tori. Finally, connected algebraic subgroups of
          $\Aut(\A^1_*\times\A^1)$ are isomorphic to a semidirect
          products of tori and commutative unipotent groups.  One
          checks that all these groups are defined over the field
          $\Q$.
	\end{proof}
	
		\begin{lem}\label{degen2}
		Let  $\varphi$ be a group automorphism of $\Aut(\A^1\times\A^1_*)$ and let $G\subset\Aut(\A^1\times\A_*^1)$ be an algebraic subgroup. Then $\varphi(G)$ is again an algebraic subgroup isomorphic to $G$.
	\end{lem}
	
	\begin{proof}
		It is enough to show the claim for $G$ connected.
		We have that $\Aut(\A^1_*\times\A^1)\simeq (\Z/2\Z\ltimes\GM)\ltimes (\C[t,t^{-1}]^*\ltimes\C[t,t^{-1}])$, where the factor $\Z/2\Z\ltimes\GM$ corresponds to the automorphisms of $\A_*^1$ and the factor $(\C[t,t^{-1}]^*\ltimes\C[t,t^{-1}])$ corresponds to the automorphisms of $\A^1_*\times\A^1$ that preserve the $\A^1$-fibration fiberwise. Note that all the subgroups in $\Aut(\A^1_*\times\A^1)$ that are isomorphic to $\GA$   belong to the factor $G^U:=\C[t,t^{-1}]$. In other words, all faithful $\GA$-actions are equivalent.
		
		By Lemma~\ref{autauttorus}, there exists a field automorphism $\tau$ of $\C$ such that the restriction of $\varphi\circ\tau$ to any maximal torus, and hence also to any subgroup isomorphic to $\GM$, is an algebraic morphism, as well as the restriction of  $\varphi\circ\tau$ to any root subgroup with respect to a maximal torus. We can write the $G^U$ as an ascending chain of algebraic subgroups $G^U_d$, each of them generated by finitely many root subgroups $U_1,\dots, U_{k(d)}$. By Lemma~\ref{oneparalg}, the restriction of $\varphi\circ\tau$ to any of the algebraic subgroups $G^U_d$ is an algebraic homomorphism and therefore the restriction of $\varphi\circ\tau$ to any algebraic subgroup of $G^U$ is an algebraic homomorphism. In particular, the restriction of $\varphi\circ\tau$ to any subgroup of $\Aut(\A^1_*\times\A^1)$ isomorophic to $\GA$ is an algebraic homomorphism. Let $G\subset \Aut(\A^1_*\times\A^1)$ be a connected algebraic subgroup. Then there exist algebraic subgroups $V_1,\dots, V_k\subset G$ each isomorphic to $\GA$ or to $\GM$ such that $G=V_1V_2\dots V_k$. Lemma~\ref{oneparalg} now implies that $\varphi\circ\tau(G)$ is an algebraic subgroup and the restriction of $\varphi\circ\tau$ to $G$ is an algebraic homomorphism. Hence, $\varphi(G)$ is an algebraic group, which is, by Lemma~\ref{basechangelemma}, isomorphic to $G$ as an algebraic group.
	\end{proof}

	\begin{proof}[Proof of Proposition \ref{groupautomorphismtoric}]
		If $S$ is isomorphic to the algebraic torus, then the claim follows from Lem\-ma~\ref{autdeg}. If $S$ is isomorphic to $\A^1\times\A^1_*$, the proposition is covered by Lemma~\ref{degen2}.
		So we can assume that 
		$S$ is isomorphic to $X_{d,e}$ for some $d,e$.
		
		By Lemma~\ref{autauttorus}, there exists a field automorphism $\tau$ of $\C$ such that the restriction of $\varphi\circ\tau$ to any maximal torus induces an isomorphism of algebraic groups between any maximal torus $T$ and its image. Now, consider the subgroup $\mathcal{N}^+_{d,e}/G_{d,e}\subset\Aut(X_{d,e})$ (see definition above). We have that  $\mathcal{N}^+_{d,e}=T\ltimes U$, where $U$ is an infinite-dimensional unipotent group. Therefore, we can write $\mathcal{N}^+_{d,e}$ and thus also $\mathcal{N}^+_{d,e}/G_{d,e}\subset\Aut(X_{d,e})$ as the limit of an ascending chain of algebraic subgroups, i.e. there exist algebraic subgroups $A_k\subset \mathcal{N}^+_{d,e}/G_{d,e}$ such that 
		\[
		\mathcal{N}^+_{d,e}/G_{d,e}=\bigcup_{k\geq 0} A_k,
		\]
		 Moreover, note that each of the groups $A_k$ is generated by the maximal torus $T/G_{d,e}$ and finitely many root subgroups $U_1,\dots, U_{l(k)}$ with respect to $T/G_{d,e}$ such that every element in $A_k$ can be written as a product of at most $C(k)$ elements from the $U_i$ or from $T/G_{d,e}$  for some $C(k) \in \mathbb{N}$. It follows from Lemma~\ref{autauttorus} that the restriction to any of these $U_i$ is an isomorphism of algebraic groups to its image. This implies that there are finitely many algebraic subgroup $V_1,\dots, V_r$ of $A_k$ such that $V_1\cdot\dots\cdot V_r=A_k$ and such that the restriction of $\varphi\circ\tau$ to any $V_i$ is an isomorphism of algebraic groups to its image. It follows that  $\varphi\circ\tau(V_1)\cdot\dots\cdots\varphi\circ\tau(V_r)$ is an algebraic subgroup and, by Lemma~\ref{oneparalg}, the restriction of $\varphi\circ\tau$ to $A_k$ is an algebraic isomorphism to its image. A similar argument applies to the factor $N_{d,e}/G_{d,e}$.
		
		Let now $G$ be any algebraic subgroup of $\Aut(S)$. Then $G$ is conjugate to a subgroup of one of the factors of the amalgamated product (\cite[Theorem 4.17]{MR3089030}). Therefore, $G$ is conjugate to a subgroup of one of the $\Jonq^\pm(\A^2)_{\leq k}/G_{d,e}$, or to a subgroup of $N_{d,e}/G_{d,e}$ respectively. Hence, the restriction of $\varphi\circ\tau$ to $G$ is an isomorphism of algebraic groups to its image. By Lemma~\ref{basechangelemma}, the algebraic group $\tau(G)$, which is just the algebraic group obtained by base-changing the field $\C$ by $\tau$ over $\Q$, is isomorphic to $G$ as an algebraic group. 
	\end{proof}

	\begin{rem}
		In general, not all automorphisms of $\Aut(S)$, where $S$ is a toric surface, are inner up to field automorphisms. For instance, let $e>1$ and $d>1$ be integers such that $e^2\not\equiv 1\mod d$ and let $X_{d,e}=\A^2/G_{d,e}$ be the corresponding toric surface. Then 
		\[
		\Aut(X_{d,e})\simeq\mathcal{N}_{d,e}/G_{d,e},
		\]
		where
		\[
		\mathcal{N}_{d,e}=\mathcal{N}_{d,e}^+*_T\mathcal{N}_{d,e}^-
		\]
		(see \cite[Lemma 4.3 and Proposition 4.4]{MR3089030}).
		We note that $\mathcal{N}_{d,e}^+\simeq T\ltimes U$, where 
		\[
		U=\{(x +p(y),  y)\mid  p(y)\in R \},
		\]
		for some subspace $R\subset \C[y]$ (here, we consider $\C[y]$ as a $\C$-vector space).  The torus induces a linear, locally finite action on $R$ by conjugation. Let $\{r_i\}$ be a basis of $R$ such that each $r_i$ is an eigenvector of $T$. We now define an automorphism $\varphi$ of $\mathcal{N}_{d,e}^+$ by setting $\varphi|_T=\id_T$ and $\varphi(x +r_i,  y)=(x+c_ir_i,y)$ for some constants $c_i\in\C^*$. Because of the amalgamated product structure of $\mathcal{N}_{d,e}$, and since $\varphi|_T=\id_T$, we can extend $\varphi$ to an automorphism $\Phi$ of $\mathcal{N}_{d,e}$ by setting $\Phi|_{\mathcal{N}_{d,e}^-}=\id_{\mathcal{N}_{d,e}^-}$. We observe that for a general choice of the constants $c_i$ the automorphism $\Phi$ is not inner, because the torus $T$ coincides with its centralizer.
		Moreover, $\Phi(G_{d,e})=G_{d,e}$, so $\Phi$ descends to
		an automorphism $\Phi'$ of $\Aut(X_{d,e})$ and again, it is straightforward to check that for a generic choice of the $c_i$,  the automorphism $\Phi'$ is not inner. Since the restriction of $\Phi'$ to the torus is the identity automorphism, $\Phi'$ will not become inner after a base-change by a field-automorphism. In fact, by Proposition~\ref{groupautomorphismtoric}, the restriction of $\Phi'$ to any algebraic subgroup is an algebraic morphism.
	\end{rem}

	\section{Proofs of the main theorems}
	
	\subsection{Toric surfaces}
	
	We first prove Theorem~\ref{toricthm}.
	Let $G$ be a linear algebraic group, $g\in G$ an element in $G$ and $\overline{\langle g\rangle}^\circ$ the neutral component of the closure of the group generated by $g$ in $G$. Then $\overline{\langle g\rangle}^\circ\simeq \GM^n\times\GA^m$, where $m\in\{0,1\}$. This follows from the structure of commutative linear algebraic groups and the fact that any unipotent group topologically generated by one element is trivial or isomorphic to $\GA$.  In what follows we will use this easy observation frequently. 	 Also recall that $\Aut(\mathbb{A}^1_* \times \mathbb{A}^1_*)\simeq \GL_2(\Z)\ltimes \GM^2$, where $\GL_2(\Z)$ is the group of monomial transformations.

	\begin{lem}\label{exceptions}
		Let $S$ be an affine normal surface. If 
		$\Aut(S)$ is isomorphic to  $\Aut(\mathbb{A}^1_* \times \mathbb{A}^1_*)$ as an abstract group, then
		$S$ is isomorphic to $\mathbb{A}^1_* \times \mathbb{A}^1_*$ as a variety.
	\end{lem}
	\begin{proof}  Let  $\phi\colon \Aut(\mathbb{A}^1_* \times \mathbb{A}^1_*) \xrightarrow{\sim} \Aut(S)$ be an isomorphism of groups and
		let $T \subset \Aut(\mathbb{A}^1_* \times \mathbb{A}^1_*)$ be 
		the maximal torus. Since $T$ coincides with its centralizer (see Lemma \ref{toruscentralizer}), 
		$\phi(T) \subset \Aut(S)$ is a closed subgroup by Lemma \ref{closed}.
		Let $d \in T$ be an element of infinite  order. Then,
		by Proposition  \ref{alg},
		$\phi(d)$  is an algebraic element of $\Aut(S)$. Hence  $\overline{\langle \phi(d) \rangle} \subset \phi(T)$ 
		is a commutative algebraic subgroup of positive dimension.  
		
		If $S$ is a toric surface, the claim of the statement
                follows from Proposition \ref{tor}.  So assume that
                $S$ is not toric. In this case, Corollary
                \ref{toricma} implies that $S$ does not admit a
                faithful action of $\GM^2$ or of $\GA \times
                \GM$. Hence, for all algebraic elements
                $h\in\Aut(\mathbb{A}^1_* \times \mathbb{A}^1_*)$ of
                infinite order the algebraic subgroup
                $\overline{\langle \phi(h)
                  \rangle}^\circ\subset\Aut(S)$ is either isomorphic
                to $\GM$ or to $\GA$.  Let
                $t_1,t_2 \in T$ be elements of infinite order with distinct centralizers in $\Aut(A^1_*\times\A^1_*)$. Then the
                commutative groups
                $\overline{\langle \phi(t_1) \rangle}^\circ$ and
                $\overline{\langle \phi(t_2) \rangle}^\circ$ do not
                coincide. Since, by assumption, $S$ is not toric,
                $\overline{\langle \phi(t_1) \rangle}^\circ$ and
                $\overline{\langle \phi(t_2) \rangle}^\circ$ are both
                isomorphic to $\GA$, because otherwise
                Corollary~\ref{toricma} implies again that $S$ is
                toric.  If
                $\overline{\langle \phi(t_1) \rangle}^\circ$ and
                $\overline{\langle \phi(t_2) \rangle}^\circ$ have
                different orbits, the algebraic group generated by
                $\overline{\langle \phi(t_1) \rangle}^\circ$ and
                $\overline{\langle \phi(t_2) \rangle}^\circ$ is
                unipotent and acts with an open orbit on $S$, which
                implies that $S \simeq \mathbb{A}^2$. This contradicts
                our assumption.  Hence, all the one-dimensional
                unipotent algebraic subgroups of $\phi(T)^\circ$ have
                the same orbits and therefore the same ring of
                invariants, which we denote by
                $\mathcal{O}(S)^{\phi(T)^\circ}$.  Consider the
                quotient map $\pi\colon S\to C$, where
                $C:=\spec(\mathcal{O}(S)^{\phi(T)^\circ})$ is an
                affine curve. Denote by
                $\Gamma:=\phi(\GL_2(\Z))\subset\Aut(S)$ the image of
                the group of monomial transformations. Since $\Gamma$
                normalizes $\phi(T)^\circ$, we obtain an action of $\Gamma$
                on $C$ which is equivariant with respect to $\pi$. Let
                $K\subset\Gamma$ be the kernel of this action. In
                particular, $K$ stabilizes the fibers of $\pi$ and
                therefore acts faithfully on the general fiber of
                $\pi$, which is isomorphic to $\A^1$. This implies
                that $K$ is solvable. Since the automorphism group of
                an affine curve is solvable or finite, we obtain that
                $\Gamma$ is solvable up to finite index. But this is a
                contradiction, since $\GL_2(\Z)$ does not contain a
                solvable subgroup of finite index.
	\end{proof}

	The following result is the analogue of Theorem~\ref{toricthm} for the particular case of the affine plane. In this case the normality hypothesis can be removed:
	
	\begin{prop}\label{A2}
		Let $S$ be an affine surface such that $\Aut(S)$ and
		$\Aut(\A^2)$ are isomorphic as groups. Then $S$ is
		isomorphic to $\A^2$.
	\end{prop}

	\begin{proof}
		Let
		$\Tr=\{ (x+c,y+d) \mid c,d \in \mathbb{C} \} \subset \Aut(\mathbb{A}^2)$
		be the subgroup of translations.  The group $\Tr$ coincides with its
		centralizer in $\Aut(\mathbb{A}^2)$. Hence, $\phi(\Tr) \subset \Aut(S)$ is a closed subgroup
		by Lemma~\ref{closed}.  The maximal torus $T\subset \Aut(\mathbb{A}^2) $ given  by the group of diagonal automorphisms
		acts on $\Tr$ by conjugation with finitely many orbits.  By Lemma~\ref{tor},
		$\phi(T) \subset \Aut(S)$ is a maximal torus.  The closed subgroup $\phi(T) \subset \Aut(S)$ also acts on $\phi(\Tr)$ faithfully and
		with finitely many orbits and hence $\phi(\Tr) \subset \Aut(S)$ is an
		algebraic subgroup of dimension $2$.  Since
		$\Tr$ does not contain elements of
		finite order, the group $\phi(\Tr)$ is unipotent. 
		If two different $\GA$-actions in $\phi(\Tr)$ have different generic orbits,
		then $\phi(\Tr)$ acts with an open orbit and because $\phi(\Tr)$ is unipotent
		it follows that $S \simeq \mathbb{A}^2$. Now assume that all the $\GA$-actions from $\phi(\Tr)$ have the same generic orbits. Then the ring of invariants $\mathcal{O}(S)^{\phi(\Tr)}$ contains non-constant functions and there exists a locally nilpotent derivation $\delta$ such that every $\GA$-action in $\phi(\Tr)$ is of the form $\{ \exp(c g\ddd) \mid c \in \mathbb{C} \}$ for some $g\in\mathcal{O}(S)^{\phi(\Tr)}$. But in this case for any $k\geq 0$ and any $f\in\mathcal{O}(S)^{\phi(\Tr)}$, the $\GA$-action  $\{ \exp(c f^k \ddd) \mid c \in \mathbb{C} \}$ commutes with all the $\GA$-actions in $\phi(\Tr)$ which implies that the centralizer of $\phi(\Tr)$ is infinite-dimensional.
		This contradicts the fact that $\dim \phi(\Tr)$ has dimension two and the claim follows.
	\end{proof}

	\begin{proof}[Proof of Theorem \ref{toricthm}]
		Let $\phi\colon \Aut(S_1) \xrightarrow{\sim} \Aut(S_2)$ be an
		isomorphism of groups and fix a maximal torus $T_1\subset\Aut(S_1)$.
		If $S_1$ or $S_2$ are isomorphic to
		$\mathbb{A}^1_* \times \mathbb{A}^1_*$, to
		$\mathbb{A}^1_* \times \mathbb{A}^1$, or to $\A^2$ then the claim
		follows from Lemma \ref{exceptions}, Proposition~\ref{tor}, or
		Proposition \ref{A2} respectively.  Now let $S_1$ be isomorphic to some
		$X_{d,e}$ different from $\A^2$, i.e. $d>1$ and $e>0$. By Proposition \ref{tor} and Proposition \ref{A2}, $S_2$ is a
		toric surface $X_{\tilde{d},\tilde{e}}$ for some $\tilde{d} > 1$ and $\tilde{e} > 0$. Moreover, $T_2 = \phi(T_1)$
		is a $2$-dimensional torus.
		
		By Lemma \ref{root}, all the root subgroups of $\Aut(X_{d,e})$ with
		respect to $T_1$ are mapped by $\varphi$ to root subgroups of
		$\Aut(X_{d',e'})$ with respect to $T_2$. Hence, to conclude the
		proof, it is enough to show that we can recover $X_{d,e}$ from the abstract group structure of its
		root subgroups and their relationship with the torus.  Assume that the torus $T_1$ acts on a root subgroup $\GA$ with weight character    $\chi$. The center of the semidirect product $\GA\rtimes_{\chi} T_1$ is exactly $\{0\}\rtimes_{\chi} \ker\chi$. The image $\varphi(\GA)$ is a root subgroup of $\Aut(S_2)$ with respect to the torus $\varphi(T_1)=T_2$ with some weight $\chi_2$. Hence the kernel of $\chi$ is mapped under $\varphi$ to the kernel  of $\chi_2$. We consider now the kernels of the  weight characters of two root subgroups with different distinguished rays and look at their intersection. More precisely, let $\chi^{\alpha_1}$ be a character with distinguished ray $\rho_1$ and $\chi^{\alpha_2}$ a character with distinguished ray $\rho_2$. By Lemma~\ref{weight-toric-surface} we have, in the notation of Section~\ref{toricsurf},
		\[
		\alpha_1=-\beta_2+l\cdot\beta_1\mbox{ and }
		\alpha_2=(a\beta_1+e'\beta_2)+k\cdot(e\beta_1+d\beta_2)\,,
		\]
		where $l,k\geq 0$, and $e',a$ are the only positive integers
		with $0 \leq e'<d$ such that $(e',d)=1$ and $ee'=1+ad$. Define
		\[
		K_{l,k}:= \ker \chi^{\alpha_1}\cap \ker \chi^{\alpha_2}\,.
		\]
		From $\chi^{\alpha_1}=1$ we obtain
		$\chi^{\beta_2}=(\chi^{\beta_1})^l$ and substituting this last equation into
		$\chi^{\alpha_2}=1$ we obtain
		$(\chi^{\beta_1})^{a+l\cdot e'+k\cdot e+lk\cdot d}=1$. This last equation has exactly $a+l\cdot e'+k\cdot e+lk\cdot d$ solutions for $\chi^{\beta_1}$. This yields
		that the order of $K_{l,k}$ is 
		\[
		|K_{l,k}|=a+l\cdot e'+k\cdot e+lk\cdot d\,.
		\]
		We fix a character $\chi^{\alpha_1}$ of $\Aut(S_1)$ with
		distinguished ray $\rho_1$. Now we consider all the characters
		$\chi^{\alpha_2}$ with distinguished ray $\rho_2$. These are exactly
		the characters corresponding to the root subgroups that do not
		commute with the root subgroup corresponding to
		$\chi^{\alpha_1}$. By considering the intersections of the kernels,
		we obtain a sequence of integers $\{|K_{l,k}|\}_{k\in\Z_{\geq 0}}$,
		where $l$ is fixed and $k$ varies. We observe that this sequence is
		an arithmetic progression. By varying $l$, we obtain a set of such
		arithmetic progressions. The smallest common difference of these
		arithmetic progressions is $d$ and the second smallest common
		difference is $d+e$.
		
		Analogously, for every fixed $k\in \Z_{\geq 0}$ the sequence of
		integers $\{|K_{l,k}|\}_{l\in\Z_{\geq 0}}$ is an arithmetic
		progression for every $k\in\Z_{\geq 0}$. The smallest common
		difference of these arithmetic progressions is $d$ and the second
		smallest common difference is $d+e'$.
		
		Since $\varphi(K_{l,k})$ is again the intersection of the kernels of
		two non-commuting characters, the sequences of integers $|K_{l,k}|$
		for $l$ or $k$ fixed are the same for $\Aut(S_1)$ and $\Aut(S_2)$.
		
		This yields that the isomorphism
		$\phi\colon \Aut(S_1) \xrightarrow{\sim} \Aut(S_2)$ can only exist
		if $S_1=X_{d,e}$ or $S_1=X_{\tilde{d},\tilde{e}}$ with
		$\tilde{d}=d$, and $\tilde{e}=e$ or if $\tilde{e}=e'$. This is,
		$S_1$ is isomorphic to $S_2$ by Lemma~\ref{isom-toric-surfaces}.
	\end{proof}

	\subsection{Non-toric $\GM$-surfaces}
  Surfaces admitting two non-commuting $\GA$-actions are
called \emph{Gizatullin} surfaces.
	\begin{rem}\label{remma}
          Let $S$ be a Gizatullin
          surface which is not toric and let $H \subset \Aut(S)$ be a
          one-dimensional algebraic subgroup.  If $H \simeq \GM$, then
          by Corollary~\ref{toricma} the only positive-dimensional
          connected algebraic subgroup in the centralizer of $H$ is
          $H$.  If $H \simeq \GA$, then again by
          Corollary~\ref{toricma}, any positive-dimensional algebraic
          subgroup in the centralizer of $H$ is isomorphic to $\GA$
          and is {\it equivalent} to $H$, i.e. it has the same  generic orbits
          as $H$, since otherwise $S$ would be isomorphic to $\A^2$.
	\end{rem}

	The next proposition shows that the property of being
	Gizatullin is encoded in the automorphism group.

	\begin{prop}\label{Gizatullin} %
		Let $S$ and $S'$ be two normal affine surfaces such that $\Aut(S)$ and $\Aut(S')$ are isomorphic as groups. Then the surface $S$ admits two non-commuting $\GA$-actions if and only if
		$S'$ admits two non-commuting $\GA$-actions.  
	\end{prop}

	\begin{proof}
		Let $S$ and $S'$ be normal surfaces with $S$ Gizatullin and let
		$\varphi\colon\Aut(S)\rightarrow\Aut(S')$ be an isomorphism of groups. We can
		assume that $S$ and $S'$ are not toric since otherwise the result follows
		directly from Theorem~\ref{toricthm}. We assume now that $S'$ is not
		Gizatullin and we derive a contradiction.
		
		Let $u$ and $v$ be nontrivial elements of two
                non-commuting $\GA$-actions on $S$.  This implies in
                particular that $u^k$ and $v^l$ do not commute for all
                $k,l>0$.   By Proposition
                \ref{alg}, $\phi(u)$ and $\phi(v)$ are algebraic
                elements in $\Aut(S')$ and so
                $\overline{\langle \phi(u) \rangle}^\circ$ and
                $\overline{\langle \phi(v) \rangle}^\circ$ are
                algebraic subgroups of $\Aut(S')$, which do not
                commute. Each of the groups
                $\overline{\langle \phi(u) \rangle}^\circ$ and
                $\overline{\langle \phi(v) \rangle}^\circ$ is
                isomorphic to $\GM^k\times\GA^l$ for some $k\geq 0$
                and some $l\in\{0,1\}$.  By Corollary~\ref{toricma} we
                can assume that the algebraic groups
                $\overline{\langle \phi(u) \rangle}$ and
                $\overline{\langle \phi(v) \rangle}$ are
                one-dimensional since otherwise $S'$ is toric.
		
		If both groups $\overline{\langle \phi(u) \rangle}^\circ$ and
		$\overline{\langle \phi(v) \rangle}^\circ$ are isomorphic to $\GA$,
		then $S'$ is Gizatullin by definition. Assume now that
		$\overline{\langle \phi(u) \rangle}^\circ \simeq \GM$ and
		$\overline{\langle \phi(v) \rangle}^\circ \simeq \GA$. Note that all the
		$\GA$-actions on $S'$ commute and have the same   generic orbits, since
		otherwise $S'$ would again be Gizatullin. By
		Lemma~\ref{ML},
		there exists a $\GA$-action on $S'$ which is normalized by
		$\overline{\langle \phi(u) \rangle}^\circ$.  Hence, there exists an
		element $s \in \Aut(S)$  such that
		$\overline{ \langle \phi(s) \rangle }$ is isomorphic to $\GA$ and is
		normalized by $\overline{ \langle \phi(u) \rangle }^\circ$.  From
		Remark~\ref{remma} it follows that the algebraic subgroup
		$\overline{\langle s \rangle}^\circ \subset \Aut(S)$ is isomorphic
		to $\GA$. It follows that, up to a different choice of $s$, we can
		assume that $\overline{\langle s \rangle}$ is connected. Since
		$\overline{ \langle \phi(s) \rangle }$ and
		$\overline{ \langle \phi(u) \rangle }$ do not commute, up to
		changing the generator $u$ of
		$\overline{ \langle \phi(u) \rangle }$, we can asume that
		$\phi(u)^{-1} \circ \phi(s) \circ \phi(u) = \phi(s^2)$.  Note that
		even after this new choice, $u$ and $s$ remain unipotent
		elements. Hence, $u^{-1} \circ s \circ u = s^2$, which means that
		the $\GA$-action $\overline{\langle u \rangle}$ acts nontrivially on
		$\overline{\langle s \rangle}\simeq \GA$, but this is not possible.
		Therefore, we are left with the case where both the groups
		$\overline{\langle \phi(u) \rangle}^\circ$ and
		$\overline{\langle \phi(v) \rangle}^\circ$ are isomorphic to
		$\GM$.
{	In this case, by \cite[Theorem~3.3]{MR2126657}, there exists a
		$\GA$-action $W$ in $\Aut(S')$ and by Lemma~\ref{ML} we can assume that $W$ is 
		normalized by
		$\overline{\langle \phi(u) \rangle}^\circ$.  In particular,   after possible scaling $u$,  we can assume that
		$\phi(u) \circ w \circ \phi(u)^{-1} = w^2$
		for a $w \in W$ of infinite order.
		Similarly as above, we
		obtain that $\overline{\langle u \rangle} \simeq \GA$ acts
		nontrivially on $\overline{\langle \varphi^{-1}(w) \rangle}$, which is again not
		possible. This proves the theorem.}
	\end{proof}

	We now prove Theorem~\ref{dynamical} that will also be used as a tool
	in subsequent proofs.
	
	\begin{proof}[Proof of Theorem~\ref{dynamical}] Let 
    $T$ be a one-dimensional torus of $\Aut(S)$ and let
	$t \in T$ be of
		infinite order. By Proposition~\ref{alg}, the image
		$\varphi(t) \in \Aut(S')$ is an algebraic element.  By
		Corollary~\ref{toricma}, the group
		$\overline{\langle \phi(t) \rangle}^\circ$ can not be isomorphic to
		$\GM\times\GM$ nor to $\GM\times\GA$ since $S'$ is non-toric, by Theorem~\ref{toricthm}. It
		follows therefore that $\overline{\langle \phi(t) \rangle}^\circ$ is
		either isomorphic to $\GM$ or to $\GA$.  By assumption, $\Aut(S)$ contains a $\GA$-action and hence, by
Lemma~\ref{ML},
		there exists a root subgroup $U$ of $S$ with respect to $T$. For any $u \in U\setminus\{\id\}$ the algebraic subgroup
		$\overline{\langle \phi(u) \rangle}^\circ$ is also isomorphic to $\GM$ or to $\GA$.  We can choose $t$ and $u$ in such a way
		that $t \circ u \circ t^{-1} = u^2$.  This implies that
		$\overline{\langle \phi(t) \rangle}^\circ$ acts nontrivially on
		$\overline{\langle \phi(u) \rangle}^\circ$ by conjugation. Hence, $\overline{\langle \phi(t) \rangle}^\circ \simeq \GM$
		and $\overline{\langle \phi(u) \rangle}^\circ \simeq \GA$.
		
		We claim that $S$ and $S'$ are both hyperbolic or both parabolic
		which will prove the theorem, since elliptic surfaces with a
		$\GA$-action are all toric, see Section~\ref{sec:non-toric}.  If $S'$ has two non-commuting root subgroups
		then $S'$ is Gizatullin and hence $S$ is Gizatullin by
		Proposition~\ref{Gizatullin}. This is only possible in the
		hyperbolic case by \cite[Corollary 4.4]{MR2196000}. Hence, in this
		case both $S$ and $S'$ are hyperbolic. We assume now that all root
		subgroups in $\Aut(S')$ commute. Furthermore, by
		Proposition~\ref{Gizatullin} we can also assume that $S$ is not
		Gizatullin and also that all root subgroups in $\Aut(S)$ commute.  We
		will now prove that if $S$ has root subgroups of different weights,
		then so does $S'$. This will prove the claim by
		Lemma~\ref{different-weight}.
		
		Assume that there  are two root subgroups $U$ and $U'$  in $\Aut(S)$ with
		different non-zero weights. Let $t\in T$ be such that
		$t\circ u\circ t^{-1}=u^2$, for all $u\in U$. Since the weights are
		different, we have that $t\circ u'\circ t^{-1}\neq (u')^2$, for all
		$u'\in U'$ different from the identity.  Let $u\in U$ be different from
		the identity. Up to changing $u$ by {a}  multiple we can assume that
		$\overline{\langle \phi(u) \rangle}$ is connected and
		hence is a root subgroup with respect to
		$\overline{\langle \phi(t) \rangle}^\circ$. Let $k\in \Z$ be such
		that $\varphi(t^k)\in \overline{\langle \phi(t) \rangle}^\circ$. Now
		we have
		$\varphi(t^k)\circ\varphi(u)\circ\varphi(t^{-k})= \varphi(u^{2^k})$.
		Similarly, let $u'\in U'$  be different from the identity. Again we can
		assume that $\overline{\langle \phi(u') \rangle}$ is
		connected and hence is a root subgroup with respect to
		$\overline{\langle \phi(t) \rangle}^\circ$, but this time we have
		\[\varphi(t^k)\circ\varphi(u')\circ\varphi(t^{-k})
		\neq\varphi\left((u')^{2^k}\right)\,,\] and so the  weights of
		$\overline{\langle \phi(u) \rangle}^\circ$ and
		$\overline{\langle \phi(u') \rangle}^\circ$ are different.
	\end{proof}

	\begin{thm}\label{semidirect-product}
		Let $S$ and $S'$ be normal surfaces. Assume
		that $\Aut(S)$ contains algebraic subgroups $T$ and $U$
		isomorphic to $\GM$ and $\GA$, respectively. If there exists a
		group isomorphism $\varphi\colon \Aut(S)\rightarrow\Aut (S')$,
		then the following hold:
		\begin{enumerate}
			\item The image $\varphi(T)\subset\Aut(S')$ is an algebraic
			subgroup isomorphic to $\GM$.  \label{torus}
			\item There exist root subgroups in $\Aut(S)$ and they are
			mapped to root subgroups such that their weights are preserved up to an
			isomorphism of the torus. \label{roots}
		\end{enumerate}
	\end{thm}

	\begin{proof}
		If $S$ is toric, then $(a)$ and $(b)$ follow from Theorem
		\ref{toricthm} and Proposition \ref{groupautomorphismtoric}. 
		Now assume that $S$ is non-toric. This implies, by Theorem~\ref{toricthm}, that $S'$ is also non-toric.  Let $t \in T$
                be of infinite order. By Proposition~\ref{alg}, the
                image $\varphi(t) \in \Aut(S')$ is an algebraic
                element.  By Corollary~\ref{toricma} the group
                $\overline{\langle \phi(t) \rangle}^\circ$ can not be
                isomorphic to $\GM\times\GM$ nor to $\GM\times\GA$
                since $S'$ is non-toric. Hence,
                $\overline{\langle \phi(t) \rangle}^\circ$ is
                isomorphic either to $\GM$ or to $\GA$.  The same is
                true for $\overline{\langle \phi(u) \rangle}^\circ$,
                where $u \in U\setminus\{\id\}$. By
                Lemma~\ref{ML} we can assume that $U$
                is normalized by $T$.  Since we can choose $t$ and $u$
                in such a way that $t \circ u \circ t^{-1} = u^2$, it
                follows that
                $\overline{\langle \phi(t) \rangle}^\circ$ acts
                nontrivially on
                $\overline{\langle \phi(u) \rangle}^\circ$ by
                conjugation. This implies that
                $\overline{\langle \phi(t) \rangle}^\circ \simeq \GM$
                and
                $\overline{\langle \phi(u) \rangle}^\circ \simeq \GA$.
		
		To prove part (\ref{torus})  we have to show that
		$\phi(T) = \overline{\langle \phi(t) \rangle} = \overline{\langle
			\phi(t) \rangle}^\circ$. Indeed, by Lemma~\ref{centralizer}, $T$ is a finite index subgroup
		of its centralizer $Z$ in $\Aut(S)$ and
		$\overline{\langle \phi(t) \rangle}^\circ$ is a finite index
		subgroup of its centralizer $Z'$ in $\Aut(S')$.
		Since $Z = \Cent_{\Aut(S)}(\langle t \rangle)$ and $Z' = \Cent_{\Aut(S')}(\langle \varphi(t) \rangle)$ we have that 
		$\varphi(Z)=Z'$. 
		We claim that the only divisible elements in $Z$ are those which belong to $T$. Indeed, $T$ is the normal subgroup of $Z$ and we can consider  the quotient map 
		$Z \to Z/T$. Since the group $Z/T$ is finite and an element from $Z/T$ is divisible if and only if it is the identity, the claim follows.  Analogously, the only divisible elements in $Z'$ belong to
		$\overline{\langle \varphi(t) \rangle}^\circ$.
		Now, since $\varphi$ maps divisible elements to divisible, we conclude that $\varphi(T) = \overline{\langle \varphi(t) \rangle}^\circ$.
		This proves $(a)$.
		
		To prove part (\ref{roots})  we first remark that by Lemma~\ref{ML}
		there exist root subgroups of $\Aut(S)$ with respect to $T$.  
		By Lemma \ref{root} root subgroups are maped to root subgroups. To conclude the proof we remark that the center of $T\ltimes_\chi U$ is
		$C=\{0\}\times \ker\chi$ and so  $\chi(t)=t^k$ or $\chi(t)=t^{-k}$, where $|k|$ is the order of
		$C$. Since the center of a subgroup is preserved under isomorphism $\varphi$, the
		proof follows.
	\end{proof}

	In the following proposition we prove part of
	Theorem~\ref{main} in the particular case, where $G$ is a $1$-dimensional connected algebraic
	group.

	\begin{prop} \label{one-dimensional}
		Let $S$ and $S'$ be normal affine surfaces and assume that there
		exists a group isomorphism
		$\varphi\colon \Aut(S) \rightarrow \Aut(S')$. Then the following
		hold:
		\begin{enumerate}
			\item The surface $S$ admits a $\GM$-action if and only if $S'$
			admits a $\GM$-action. \label{GM} 
			\item The surface $S$ admits a $\GA$-action if and only if $S'$
			admits a $\GA$-action. \label{GA}
		\end{enumerate}
	\end{prop}

	\begin{proof}
	 To
		prove our theorem, we can assume that $S$ and $S'$ are not toric
		since in that case our theorem follows directly from
		Theorem~\ref{toricthm}.  We are now going to prove part (\ref{GM}).
		Assume there exists an algebraic subgroup $T\subset\Aut(S)$ isomorphic to
		$\GM$. If $S$ admits a non-trivial $\GA$-action, then the result
		follows from Theorem~\ref{semidirect-product}. We assume in the sequel
		that $S$ admits no $\GA$-action. We will show that there exists an algebraic
		subgroup $\Aut(S')$ that is isomorphic to $\GM$.
		All elements in the image $\varphi(T)$ are algebraic by
		Theorem~\ref{alg}. If $\Aut(S')$ contains no algebraic subgroup isomorphic to $\GM$, then all elements of infinite order in
		$\varphi(T)$ are unipotent. If $\Aut(S')$ contains
		two non-commuting unipotent subgroups, then we are in the case of
		Proposition~\ref{Gizatullin}. This implies that $\Aut(S)$ also
		contains two non-commuting unipotent subgroups which is a case we
		already excluded. Hence, all unipotent subgroups in $\Aut(S')$
		commute. In particular, the set of all unipotent elements in
		$\Aut(S')$ is a normal commutative subgroup such that every
		algebraic subgroup has the same ring of invariants. Let $U \subset \Aut(S')$ be any
		algebraic subgroup isomorphic to $\GA$. The quotient map
		$\pi\colon S'\to \spec \OO(S')^U$ gives an $\A^1$-fibration which is
		preserved by $\Aut(S')$ in the sense that every element
		$g\in\Aut(S')$ permutes the fibers. A comprehensive study of the
		automorphism groups of 2-dimensional $\A^1$-fibrations was carried
		out in \cite[Section~8]{MR3793368}. Excluding the case of toric
		surfaces, we are either in the setting of \cite[Theorem~8.13]{MR3793368} or
		\cite[Theorem~8.25]{MR3793368}. In both cases we obtain that $\Aut(S')$ contains an
		algebraic subgroup isomorphic to $\GM$ because $\Aut(S')$ 
		contains infinitely many elements of finite order. This proves part
		(\ref{GM}) in the theorem.
		
		Finally, part (\ref{GA}) of the theorem follows by an exclusion
		argument. Indeed, assume $\Aut(S)$ contains a subgroup $U$
		isomorphic to $\GA$. Let $u\in U$ be of infinite order, then
		$\varphi(u)$ is algebraic and so
		$\overline{\langle\varphi(u)\rangle}$ contains a connected
		$1$-dimensional algebraic subgroup $G$. If $G\simeq \GA$, then part
		(\ref{GA}) follows. If $G\simeq \GM$, then by part (\ref{GM}) we
		have that $\Aut(S)$ contains also an algebraic subgroup isomorphic
		to $\GM$. Now the theorem follows by Theorem~\ref{semidirect-product}
		that implies that $\Aut(S')$ has a root subgroup. 
	\end{proof}

	Let $\SL_2$ and $\PSL_2$ be the special linear and projective special
	linear groups, respectively, over the base field of complex numbers.
	Surfaces with a non-trivial $\SL_2$-action are well understood, see
	for example \cite[Section 4]{MR768181} or \cite{MR0340263}. We will
	rely on these results to treat the case where $G=\PSL_2$   in our main
	result Theorem~\ref{main}.
	
	\begin{lem}[{\cite[Section 4]{MR768181}}] \label{list-SL}
		Let $S$ be a normal affine $\SL_2$-surface. Then $S$ is
		isomorphic to one of the following:
		\begin{itemize}
			\item a toric surface;
			\item $\SL_2/T$, where $T\subset \SL_2$ is a maximal torus;
			\item $\SL_2/N$, where $N\subset \SL_2$ is the normalizer of
			$T$.
		\end{itemize}
	\end{lem}

	In the next proposition we prove that normal affine $\SL_2$-surfaces
	are uniquely determined by their automorphism groups.
	
	\begin{prop}\label{SL2T}
		Let $S$ and $S'$ be normal affine surfaces such that there exists a
		group isomorphism $\varphi\colon\Aut(S)\rightarrow \Aut(S')$. Then the following holds:
		\begin{itemize}
			\item If $G\subset \Aut(S)$ is an algebraic subgroup isomorphic to $\PSL_2$, then $\varphi(G)\subset\Aut(S')$ is an algebraic subgroup isomorphic to $\PSL_2$.
			\item  The surface $S$
			is isomorphic to $\SL_2/T$ if and only if $S'$ is isomorphic to  $\SL_2/T$.
			\item The surface $S$
			is isomorphic to $\SL_2/N$ if and only if $S'$ is isomorphic to  $\SL_2/N$.
			\item 	Normal affine $\SL_2$-surfaces are
			uniquely determined by their automorphism groups.
		\end{itemize}

	\end{prop}
	
	\begin{proof}
	If $S$ is toric, the statements follow from  Theorem~\ref{toricthm} and Proposition~\ref{groupautomorphismtoric}, so we can assume that $S$ is non-toric.
	{To prove the first statement of the proposition we first note that,}
		by Theorem \ref{semidirect-product}, the isomorphism $\varphi$ maps 
		a maximal torus $D\subset\Aut(S)$ into a $1$-dimensional torus
		$\varphi(D)$. Also by Theorem~\ref{semidirect-product},  the root
		subgroups of $\Aut(S)$ with respect to $D$ are mapped to root
		subgroups of $\Aut(S')$ with respect to $\varphi(D)$ and the
		weights are preserved up to an automorphism of the torus.
		
		The group $G$ is generated by two non-commuting root subgroups $U$
		and $V$ with respect to $D$. Therefore, we obtain two
		non-commutative root subgroups $\varphi(U)$ and $\varphi(V)$ of
		$\Aut(S')$ with respect to $\phi(D)$ such that the group
		$H=\varphi(G)$ generated by $\varphi(U)$ and $\varphi(V)$ consists
		of algebraic elements.  There exists a $n$ such that every element
		in $H$ can be written as a product $g=u_1v_1\cdot\dots\cdot u_nv_n$,
		where $u_i\in \varphi(U)$ and $v_i\in \varphi(V)$. This product
		structure shows that $H$ is a constructible set with a group
		structure, and therefore $H$ is closed. Moreover, $H$ is contained in a
		filter set of the ind-group $\Aut(S')$ and hence $H$ is an algebraic
		group. This proves the first statement.
		
	The {normalizers}  of $G$ in $\Aut(\SL_2/T)$ and $\Aut(\SL_2/N)$ are different
		and so $\SL_2/T$ and $\SL_2/N$ are distinguished from each
		other by their automorphism groups.  Together with Lemma~\ref{list-SL} {and Theorem~\ref{toricthm}} this proves the last three statements.
	\end{proof}

	We now proceed to the proof of our main theorem.
	
	\begin{proof}[Proof of Theorem~\ref{main}]
		 If $S$ is toric
		then the theorem  follows directly from  Proposition \ref{groupautomorphismtoric}. We
		assume in the sequel that $S$ is non-toric and therefore we also know that $S'$ is non-toric.  The
		only connected algebraic groups $G$ acting faithfully on non-toric
		surfaces are $\GM$, $\GA^n$, $\GM\ltimes\GA^n$, and  $\PSL_2$, since
		$G$ must have rank at most $1$. Remark that $\SL_2$ is excluded from
		this list since it only acts faithfully on toric surfaces by
		Lemma~\ref{list-SL}. Furthermore, we don't consider $\GA^n$ because it is unipotent and hence excluded from the statement of the theorem.

If $G\simeq\PSL_2$, the theorem follows from
Proposition~\ref{SL2T}. If $G\simeq\GM\ltimes\GA$, the theorem follows
from Theorem~\ref{semidirect-product}. Assume now that
$G\simeq\GM$. If $G$ is contained in an algebraic subgroup
$G\ltimes H$ for some $H\subset \Aut(S)$ isomorphic to $\GA$, then the
theorem follows from the previous case. Hence, by
Theorem~\ref{semidirect-product} we can assume that there are no
unipotent elements in $\Aut(S)$. By
Proposition~\ref{one-dimensional}~(\ref{GM}) we have that $S'$ also
admits a faithful action of $\GM$. With the same argument as above, we
obtain that there are no unipotent elements in
$\Aut(S')$. Furthermore, by \cite[Theorem~3.3]{MR2126657} we have that
 $G$ is the unique subgroup of $\Aut(S)$ isomorphic to $\GM$. The
image $\varphi(G)$ is commutative and contains a 1-parameter subgroup
$G'$ isomorphic to $\GM$. By Lemma~\ref{centralizer} we have that
$\varphi(G)$ is contained in $G'\times H$ where $H$ is a finite
subgroup of $\Aut(S')$. Since, by divisibility of elements of $\GM$,
the group $\GM$ has no non-trivial finite quotients, we have that
$\varphi(G)\subset G'$. By symmetry, applying the same argument to
$\varphi^{-1}$ we obtain $\varphi^{-1}(G')\subset G$. This yields
$\varphi(G)=G'$.
		
Finally, assume that $G\simeq\GM\ltimes\GA^n$. We can write $\GA^n$ as
a direct product $U_1\times\ldots\times U_n$ of root subgroups. All
the root subgroups $U_i$ have the same generic orbits since otherwise
$S\simeq \A^2$ and we assume $S$ is not toric. Furthermore, all the
weights of $U_i$ have the same sign in the character lattice of $\GM$,
which is isomorphic to $\Z$ by Lemma~\ref{different-weight}.
By Theorem~\ref{semidirect-product} $\GM$ is mapped to $\GM$ and every $U_i$ is mapped to 
{a
root
subgroup} $U'_i \subset \Aut(S')$ {that has} the same weight {as} $U_i$, up to an automorphism
of $\GM$. Since $S'$ is also not
toric and $U_i'$ commutes with $U_j'$ for all $i,j$, again by
Lemma~\ref{different-weight} we have that all $U_i'$ have the same
sign in the character lattice of $\GM$. Hence, $\varphi(G)$ is an
algebraic group isomorphic to $G$.
\end{proof}
	
\begin{rem}\label{unipotentrem}
  Let $S$ and $S'$ be normal affine surfaces and let
  $\varphi\colon\Aut(S)\to\Aut(S')$ be an isomorphism of groups. If
  $\Aut(S)$ contains a unipotent subgroup $U$, then $\Aut(S')$ also
  contains an algebraic subgroup isomorphic to $U$.  Indeed, assume
  $U \simeq\GA^n$.  It is enough to prove that $S'$ admits a
  $\GA$-action, but this has been shown in
  Proposition~\ref{one-dimensional}.
\end{rem}

	\subsection{Higher dimensional toric varieties}
	\label{sec:higher}

	The following lemma is needed for the proof of Theorem~\ref{mainhighdim}.
	
	\begin{lem}\label{n-1}
		Let $X_\sigma$ be an affine toric variety of dimension $n$ and let
		$\TT \subset \Aut(X_\sigma)$ be a maximal torus. Then there
		exists an $(n-1)$-dimensional torus $H \subset \TT$ such that all
		root subgroups of $\Aut(X)$ with respect to $H$ have different
		weights.
	\end{lem}
	\begin{proof}
		Take any $(n-1)$-dimensional subtorus $H\subset\TT$ such that
		$N_H\cap \sigma=\{0\}$, where $N_H$ is the sublattice of $N$ of
		1-parameter subgroups of $\TT$ that are contained in $H$ and recall
		that $\sigma\subset N_\RR$. It is clear that every $\TT$-root
		subgroup is also a $H$-root subgroup. Now,
		\cite[Proposition~1]{MR3263267} shows that every $H$-root subgroup is
		also $\TT$-root subgroup and so Corollary~\ref{diff-toric} implies
		that the weights of $H$-root subgroups are also different.
	\end{proof}
	
	We now show that affine toric varieties are determined by
	their set of roots.

	\begin{lem}\label{det-roots}
		Let $\sigma$ and $\sigma'$ be strongly convex rational
		polyhedral cones in $N_\RR$. If $\RT(\sigma)=\RT(\sigma')$
		then $\sigma=\sigma'$. In particular, a toric variety
		$X_\sigma$ is completely determined by the set of its roots
		with respect to any fixed maximal torus in $\Aut(X_\sigma)$.
	\end{lem}
	\begin{proof}
		To prove the lemma, it is enough to show that the cone
		$\sigma$ of a toric variety $X_\sigma$ can be recovered from
		the set of roots $\RT(\sigma)$. Since any strongly convex
		rational polyhedral cone is the convex hull of its rays, it
		is enough to show that every ray $\rho\in \sigma(1)$ can be
		recovered from the set of roots. By \cite[Remark~2.5]{MR2657447}
		the set $\RT_\rho(\sigma)$ of roots with $\rho\in\sigma(1)$
		as distinguished ray is not empty. Hence, to recover
		$\sigma$ from $\RT(\sigma)$ it is enough to recover for
		every $e\in \RT(\sigma)$ its distinguished ray.
		
		By \cite[Remark~2.5]{MR2657447}, the lattice vector
		$m+e\in \RT_\rho(\sigma)$ for every $e\in \RT_\rho(\sigma)$
		and every $m\in \rho^\bot\cap \sigma^\vee_M$. Let us fix now
		a root $e\in \RT(\sigma)$. By the preceding consideration,
		there exists a hyperplane $H\subset M_\RR$ such that the
		linear span of $H\cap(\RT(\sigma)-e)$ equals $H$. Take now
		$L=H^\bot\subset N_\RR$ the line orthogonal to $H$. The line $L$ is
		composed of two rays and has only two primitive vectors
		$\pm p\in L$. The distinguished ray of $e$ is given by
		$\rho_e=-\langle e,p\rangle \cdot p$ since 
		$\langle e,\rho_e \rangle=-\langle e,p \rangle^2=-1$.
	\end{proof}

	\begin{prop}\label{propfinal}
		Let $X$ and $Y$ be affine toric varieties. If
		$\Aut(X)$ and $\Aut(Y)$ are isomorphic as ind-groups, then
		$X$ and $Y$ are isomorphic. 
	\end{prop}
	\begin{proof}
		Fix an isomorphism $\phi\colon\Aut(X)\rightarrow \Aut(Y)$ of
		ind-groups. Let $\TT \subset \Aut(X)$ be a maximal torus of
		dimension $n$. Then $\phi(T)\subset\Aut(Y)$ is a torus of dimension $n$ and therefore, $\dim(X)=\dim(Y)$.  Since the groups $\Aut(X)$ and $\Aut(Y)$ are isomorphic, the
		root subgroups of $\Aut(X)$ with respect to $\TT$ are sent to root
		subgroups of $\Aut(Y)$ with respect to $\phi(T)$. Moreover, weights
		are preserved under this isomorphism. Now, Lemma~\ref{det-roots}
		implies that $X\simeq Y$ as a toric variety. In particular,
		$X\simeq Y$ as a variety.
	\end{proof}

	We now proceed to prove the last remaining main result of this
	paper.
	
	\begin{proof}[Proof of Theorem~\ref{mainhighdim}]
		Let $\phi\colon\Aut(X) \rightarrow \Aut(Y)$ be an isomorphism of
		ind-groups and let $n$ be the dimension of $X$.  By Lemma~\ref{n-1} there is a
		subtorus $H \subset \Aut(X)$ of dimension $n-1$ such that all the root
		subgroups of $\Aut(X)$  with respect to
		$H$ have different weights. Since $\phi$ is an isomorphism of ind-groups, all the root subgroups of $\Aut(Y)$ with respect to the algebraic torus
		$\phi(H)\subset \Aut(Y)$ have
		different weights . Hence, by Lemma~\ref{Kr15} we have
		$\dim Y \leq n$. Since $X$ admits an $n$-dimensional faithful torus
		action, the same holds for $Y$ and we conclude that $Y$ is also a
		toric variety of dimension $n$. Now the theorem follows from
		Proposition~\ref{propfinal}.
	\end{proof}

\subsection{Examples and Counterexamples}\label{excounterex}
	
In the remaining of this paper we provide examples showing that
certain conditions in Theorem~\ref{main}, Theorem~\ref{dynamical}, and
Theorem~\ref{mainhighdim} are necessary for the statements to hold.
	
\begin{lem}\label{fibration}
  Let $S$ be a surface endowed with a non-constant morphism
  $\pi\colon S\rightarrow C$ to a smooth curve $C$. Assume that $C$ is
  non-rational, has trivial automorphism group, and that the general fibers of $\pi$ are
  rational. Then $\pi$ is invariant under automorphisms of $S$.
\end{lem}
	
\begin{proof}
  Let $\alpha\colon S\to S$ be an automorphism and let $F$ be a
  general fiber of $\pi$. The morphism $(\pi\circ\alpha)|_F$ is either
  dominant or constant. If it is dominant, then $C$ is rational which
  is a contradiction. Hence $\pi\circ\alpha(F)\subset C$ is a point
  and therefore $\alpha(F)$ another fiber of $\pi$. It follows that
  $\alpha$ preserves general fibers fiberwise and therefore that
  $\alpha$ preserves all the fibers. Since $C$ has trivial
  automorphism group, $\pi$ is invariant under $\alpha$. 
\end{proof}

To prove that the restriction on unipotent subgroups cannot be removed
in Theorem~\ref{main}, we have the following example, which was explained to us by Adrien Dubouloz and Hanspeter Kraft.

\begin{exa} \label{example-unipotent} %
  Let $C$ be a smooth affine curve with trivial automorphism group
  and no non-constant invertible regular functions. Let $c_0\in C$ be
  any point and let $F$ be the fiber $F=\A^1\times \{c_0\}$ inside
  $\A^1\times C$ of the projection map to $C$. Take a projective
  completion $\widetilde{S}$ of $\A^1\times C$ with boundary $H$.  Let
  $\operatorname{Bl}(\widetilde{S})$ be the surface obtained by
  blowing-up $\widetilde{S}$ at three generic points $(p_0,c_0)$, $(p_1,c_0)$, and $(p_2,c_0)$ of $F$. Furthermore, let $S$ be the
  surface obtained by removing from $\operatorname{Bl}(\widetilde{S})$
  the strict transform of $F+H$. Since $F+H$ is ample the strict
  transfrom is also ample, by Kleiman's Criterion, and so the surface
  $S$ is affine.  There is a morphism $\rho\colon S\to \A^1\times C$
  whose composition with the second projection gives an
  $\A^1$-fibration $\pi\colon S\to C$. The morphism $\rho$ restricted
  to the preimage $\pi^{-1}(C\setminus\{c_0\})$ is an isomorphism onto
  $\A^1\times (C\setminus\{c_0\})$. Furthermore, $\pi^{-1}(c_0)$
  consists of three disjoint copies of $\A^1$.
		
  Let now $\alpha$ be an automorphism of $S$. By Lemma
  \ref{fibration}, $\alpha$ preserves the fibration over $C$ fiber by
  fiber. In particular, it induces an automorphism on
  $\A^1\times (C\setminus\{c_0\})$, which is of the form
  $(x,c)\mapsto (f(c)x+g(c),c)$, where
  $f\in \mathcal{O}(C\setminus\{c_0\})^*$ and
  $g\in\mathcal{O}(C\setminus\{c_0\})$. We claim that $\alpha$ is the
  lift of an automorphism $\alpha_0$ of $\A^1\times C$, i.e.
  $f\in \mathcal{O}(C)^*=\CC^*$ and $g\in\mathcal{O}(C)$.

  Indeed, let $C'\subset C$ be a neighborhood of $C$ where the divisor
  $c_0$ is principal and let $h$ be a rational function on $C$ such
  that $c_0=\operatorname{div}(h)|_{C'}$. Let $S'$ be the preimage
  $\pi^{-1}(C')$ of $C'$ by the fibration. Then $S'$ is obtained by
  blowing-up the points $(p_0,c_0)$, $(p_1,c_0)$, and $(p_2,c_0)$
  inside $\A^1\times C'$ and removing the strict transform of $F$. Let
  now $A=\OO(C')$, $\A^1=\spec \CC[x]$ and
  $\mathbb{P}^1=\operatorname{Proj}\CC[s,t]$. Letting
  $r(x)=(x-p_0)(x-p_1)(x-p_2)$, the blow-up of the three points is
  given by the equation $ht=rs$ on $\A^1\times C'\times \mathbb{P}^1$
  and the strict transform of $F$ is given by $s=0$ in the
  blow-up. Hence, the affine surface $S'$ is naturally covered by two
  affine charts $U=\spec A[x,t]/(ht-r)$ and
  $U'=\spec A[x,s,s^{-1}]/(h-rs)$ glued along
  $U_0=\spec A[x,\tfrac{s}{t},\tfrac{t}{s}]/(h-r\tfrac{s}{t})$. But
  $U'=U_0$ and so $S'=U=\spec B$, where $B=A[x,t]/(ht-r)=A[x,rh^{-1}]\subset A[x, h^{-1}]$. Now, the
  comorphism $\alpha^*$ of the automorphism $\alpha$ maps $A$ to $A$
  identically and maps $x$ to $\alpha^*(x)=fx+g$ but $f$ and $g$ are
  regular on $C'\setminus \{c_0\}$. Hence $f=ah^\ell$ with
  $a\in \CC^*, l\in\mathbb{Z}$, and $g\in A[h^{-1}]$. We conclude that
  $\alpha^*(x)=ah^\ell x+g\in B$. A straightforward computation yields
  $(\alpha^{-1})^*(x)=a^{-1}h^{-\ell}(x-g)\in B$. Up to changing
  $\alpha$ by its inverse, we can assume $\ell\geq 0$. Let now
  $\mu_i:A[x,h^{-1}]\rightarrow A[h^{-1}]$ be the evaluation
  homomorphism $x=p_i$ for $i=0,1,2$.  Restricting $\mu_i$ to $B$ for
  any $i=0,1,2$, we obtain a homomorphism $B\rightarrow A$ that is the
  identity on $A$. This yields $B\cap A[h^{-1}]=A$. In particular,
  $\mu_i(h^{-\ell}(x-g))=h^{-\ell}(p_i-g)\in A$ for all $i=0,1,2$. This
  yields $h^{-\ell}(p_i-p_j)\in A$ for all $i,j=0,1,2$. Hence $\ell=0$
  and so $g\in A=\OO(C')$. Finally, the claim follows since
  $\OO(C\setminus\{c_0\})\cap \OO(C')=\OO(C)$.
  
  Note as well that $\alpha_0$ must preserve the set
  $(p_0,c_0)$, $(p_1,c_0)$, and $(p_2,c_0)$ since $\alpha$ preserves the
  fiber $\pi^{-1}(c_0)$ of $S$. As the points $p_i$ were chosen
  generically, it follows that $\alpha_0|_{\A^1\times\{c_0\}}$ is the
  identity. It follows that $f=1$ and that $g$ is contained in the
  ideal $\mathfrak{m}_{c_0}$ of $c_0$ in $\mathcal{O}(C)$. In
  particular, every element in $\Aut(S)$ is unipotent and $\Aut(S)$ is
  abstractly isomorphic to the additive group
  $(\mathfrak{m}_{c_0},+)$, which is abstractly isomorphic to the
  additive group $(\C,+)$.
		
  Now, let $G\subset\Aut(S)$ be any algebraic subgroup isomorphic to
  $\GA$. We can find an automorphism $\varphi$ of $\Aut(S)$ such that
  $\varphi(G)$ is dense in $\Aut(S)$. In particular, $\varphi(G)$ is
  not an algebraic subgroup.
\end{exa}

In the next two examples we exhibit elliptic and hyperbolic
$\GM$-surfaces, respectively whose automorphism groups are isomorphic
and isomorphic to $\GM$. This shows that the hypothesis to have
additionally a $\GA$-action in \ref{dynamical} is essential.
	
\begin{exa} \label{example-elliptic}
		
  Let $C$ be a smooth projective curve normally embedded in
  $\mathbb{P}^n$. Assume further that $C$ is non-rational and has
  trivial automorphism group. Let $S$ be the affine cone in
  $\mathbb{A}^{n+1}$ of $C$. The group $\GM$ is contained in $\Aut(S)$
  by scalar multiplication making $S$ into an elliptic
  $\GM$-surface. The origin $\bar{0}\in S$ is a singular point in $S$
  since $S$ is different from $\mathbb{P}^1$ and so it is fixed by any
  automorphism \cite[Theorem~2.5]{MR644276}. Furthermore, there is a
  non-constant regular map $\pi:S\setminus\{\bar{0}\}\rightarrow
  C$. Hence, by Lemma~\ref{fibration} we obtain that $\pi$ is
  invariant under any automorphism of $S$. This yields
  $\Aut(S)\subseteq \GM\rtimes \mathbb{Z}/2\mathbb{Z}=\Aut(\GM)$.  Let
  now $L$ be the closure in $S$ of a fiber of $\pi$. Any automorphism
  restricts to an automorphism of $L$ fixing $\bar{0}$. This yields
  $\Aut(S)= \GM$.
\end{exa}
	
\begin{exa} \label{example-hyperbolic}

  Let $C$ be a non-rational smooth affine curve having trivial
  automorphism group {and no invertible functions}. The surface $\A^1\times C$ admits a natural
  parabolic $\GM$-action acting on the first coordinate. The set of
  fixed points is given by $F=\{0\}\times C$. Let $\widetilde{S}$ be
  an equivariant projective completion of $\A^1\times C$ with boundary
  $H$. Let $\operatorname{Bl}(\widetilde{S})$ be the surface obtained
  by blowing up $\widetilde{S}$ at a point $(0,c_0)$ in $F$ with
  ideal $(y,\mathfrak{m}_{c_0}^2)$, where $\mathfrak{m}_{c_0}$ is the
  maximal ideal of $c_0\in C$ and $\A^1=\spec\CC[x]$. Furthermore, let
  $S$ be the surface obtained by removing from
  $\operatorname{Bl}(\widetilde{S})$ the strict transform of
  $F+H$. Since $F+H$ is ample, the strict transform is also ample, by
  Kleiman's Criterion and so the surface $S$ is affine.

  Since the blown-up point in $\A^1\times C$ is fixed and $F+H$ is
  also fixed, the $\GM$-action on $\A^1\times C$ lifts to a
  $\GM$-action on $S$. Let $L$ be the line corresponding to the
  exceptional divisor in $\operatorname{Bl}(\widetilde{S})$
  intersected with $S$. The line $L$ is composed of a 1-dimensional
  orbit $O_1$ and a fixed point $p'$. Furthtemore, there are exactly
  two orbit closures meeting at $p'$: the closure of $O_1$ and the
  closure of the strict transform $O_2$ of the only one-dimensional
  $\GM$-orbit in $\A^1\times C$ having $p$ in its
  closure. Furthermore, the algebraic quotient $\pi:S\rightarrow C$ is
  induced from the second projection on $\A^1\times C$. By
  Lemma~\ref{fibration} we obtain that $\pi$ is invariant under any
  automorphism of $S$. This yields
  $\Aut(S)\subseteq \GM\rtimes
  \mathbb{Z}/2\mathbb{Z}=\Aut(\GM)$. Furthermore, if the element of
  inversion $t\mapsto t^{-1}$ in $\Aut(\GM)$ induced an automorphism
  in $S$, it would exchange the two orbits $O_1$ and $O_2$ meeting at
  the only fixed point. We claim that $S$ is hyperbolic and that the
  orbit $O_1$ has isotropy $\Z/2\Z$ by construction while the orbit
  $O_2$ has trivial isotropy. This yields $\Aut(S)= \GM$ and so $S$ is
  the desired example.

  To prove the claim we argue as in Example~\ref{example-unipotent}.
  Let $C'\subset C$ be a neighborhood of $C$ where the divisor $c_0$
  is principal and let $h$ be a rational function on $C$ such that
  $c_0=\operatorname{div}(h)|_{C'}$. Let $S'$ be the preimage
  $\pi^{-1}(C')$ of $C'$ by the natural map
  $\pi\colon S\rightarrow C$. With the same argument as in
  Example~\ref{example-unipotent} we obtain that $S'=\spec B$, where
  $B=A[x,s]/(h^2-ys)$. The $\GM$-action is given by $\deg(h)=0$ and
  $\deg(y)=1$ and so $\deg(s)=-1$. This proves that $S$ is
  hyperbolic. The orbit $O_1$ corresponds to $s=0$ and so has isotropy
  $\Z/2\Z$. On the other hand, the orbit $O_2$ corresponds to $h=0$
  and so has trivial isotropy. This proves the claim.
\end{exa}
	
Finally, we show that the assumption in Theorem~\ref{mainhighdim} that
the variety $X$ is not isomorphic to an algebraic torus, is necessary.
	
\begin{exa}\label{main2}
  Let $\TT$ be an algebraic torus and let $C$ be a smooth affine curve
  having trivial automorphism group and no non-constant invertible
  regular functions. We claim that $\Aut(\TT)$ and $\Aut(C\times\TT)$
  are isomorphic as ind-groups.  Indeed, let
  $\varphi\colon C\times \TT\rightarrow C\times \TT$ be an
  automorphism of $C\times \TT$. By Lemma~\ref{fibration} the first
  projection $\pr_1\colon C\times\TT\rightarrow C$ is invariant under
  automorphisms of $C\times\TT$. Hence, $\varphi(x,t)=(x,\psi(x,t))$
  for all $x\in C$, $t\in \TT$ and some morphism
  $\psi \colon C \times T \to T$. For every $t\in \TT$ we let
  $\psi_t\colon C\rightarrow \TT$ be the map given by
  $\psi_t(z)=\psi(z,t)$. The comorphism
  $\psi_t^*(z)\colon\OO(T)=\KK[M]\rightarrow \OO(C)$ sends invertible
  functions to invertible functions, but $\KK[M]$ is generated by
  invertible functions while $C$ admits no invertible function other
  than the constants. Hence, the image of $\psi_t^*$ is the base field
  and so the map $\varphi_t$ is constant.  We obtain that
  $\psi(z,t)=\widetilde{\psi}(t)$ for some automorphism
  $\widetilde{\psi}\colon\TT\rightarrow \TT$ of the torus $\TT$ and
  for all $z\in C$, $t\in \TT$. This yields that the automorphism
  $\varphi$ is a product $\varphi=\id_C\times \widetilde{\psi}$, which
  proves the claim.
\end{exa}

\bibliographystyle{amsalpha} %
\bibliography{bibliography_cu}

\end{document}